\documentclass[reqno,twoside,a4paper,11pt]{amsart}
\usepackage{amsmath,amssymb,dsfont,amsthm}
\usepackage{graphics}
\usepackage{xypic}
\usepackage{latexsym}
\usepackage{amsfonts}
\usepackage{geometry}
\usepackage{tensor}
\usepackage{bm,color}
\usepackage{relsize}
\usepackage{hyperref}

\usepackage{upgreek}

\hypersetup{
    colorlinks=true,
    linkcolor=blue,
    filecolor=magenta,      
    urlcolor=cyan,
    citecolor=blue,
    pdftitle={Sharelatex Example},
    bookmarks=true,
    pdfpagemode=FullScreen,
}

\date{\today}

\title[]{On the cohomology of $L^2$-harmonic forms of an incomplete Riemannian manifold}

\author{F. Bei, M. Spreafico}
\address[Francesco Bei]{\tt Dipartimento di Matematica, Sapienza Universit\`a di Roma}
\email{bei@mat.uniroma1.it}
\address[Mauro Spreafico]{\tt Dipartimento di Fisica, Universit\`a di Trento, and INFN Trento, Italy}
\email{mauro.spreafico@unitn.it}

\numberwithin{equation}{section}

\newtheorem{theo}[subsubsection]{Theorem}
\newtheorem{lem}[subsubsection]{Lemma}
\newtheorem{corol}[subsubsection]{Corollary}
\newtheorem{defi}[subsubsection]{Definition}

\newtheorem{prop}[subsubsection]{Proposition}
\newtheorem{rem}[subsubsection]{Remark}
\newtheorem{ass}[subsubsection]{Assumption}

\newcommand{\beq}{\begin{equation}}
\newcommand{\eeq}{\end{equation}}

\newcommand{\de}{{\delta}}

\newcommand{\om}{\omega}

\newcommand{\De}{\Delta}

\renewcommand{\Im}{{\rm Im\hspace{1pt}}}

\newcommand{\bu}{{\bullet}}

\renewcommand{\d}{{\rm d}}

\renewcommand{\H}{{\mathds{H}}}

\renewcommand{\H}{\mathcal{H}}

\newcommand{\CS}{\mathsf{C}}
\newcommand{\DS}{\mathsf{D}}

\newcommand{\HH}{{\mathcal{H}}}

\newcommand{\HF}{\mathfrak{H}}

\begin{document}

\maketitle
\begin{abstract} Motivated by the work of Cappell, Deturck, Gluch and Miller, we 
extend the notion of cohomology of harmonic forms (of a compact manifold with boundary) to the abstract setting of Hilbert complexes. Then, we present some geometric applications of our construction to  incomplete Riemannian manifolds with particular interest to the case of smoothly stratified Thom-Mather spaces.

%
%

\end{abstract}

\tableofcontents 

\section{Introduction}

Let $(M,g)$ be a smooth compact connected oriented Riemannian $n$-manifold, possibly with boundary, and $\Omega^\bu(M)$ the vector space of the smooth differential forms. Denoting by $d$ the de Rham differential and by $\de$ the co-differential with respect to $g$, one defines the Hodge-Laplace operator by $\Delta=\de d+d \de$, and the space of the harmonic forms $\H^\bu(M)$ as those satisfying $\De\om=0$. Since the exterior differential $d$ commutes with $\De$, we have a subcomplex $(\H^{\bullet}(M),d_{\bullet})$ of the de Rham complex $(\Omega^{\bullet}(M),d_{\bullet})$. When $M$ is a closed manifold, a form is harmonic if and only if it is closed and co-closed, whence the cohomology groups of $(\H^{\bullet}(M),d_{\bullet})$ and  $(\Omega^{\bullet}(M),d_{\bullet})$ coincide. This fact is no longer true if the boundary of $M$ is not empty. 
More precisely, Cappel, Deturck, Gluch and Miller proved that in this case there is an isomorphism 
\[
H^q(\H^\bu(M),d_{\bullet})\cong H^q(\H^\bu(M),d_{\bullet})\oplus H^{q-1}(\H^\bu(M),d_{\bullet}).
\]
see \cite{CDGM}. As one can see from the above isomorphism, the cohomology of the complex of harmonic forms is different but still completely determined by the de Rham cohomology of $M$. In particular, $H^q(\H^\bu(M),d_{\bullet})$ has a purely topological interpretation. In \cite{CDGM} the authors also observed that it would be interesting to understand to what extent these results have analogues for $L^2$-harmonic forms on smoothly stratified pseudomanifolds, see \cite[p. 925]{CDGM}. Motivated by the above question, we started to investigate how to define a complex of $L^2$-harmonic forms over an arbitrarily fixed incomplete Riemannian manifold. Note that in this setting even the definition of the complex, due to the incompleteness of the metric, is not clear at all and present interesting challenges. Indeed, one needs to define an $L^2$ extension of $\Delta_k$, the Hodge Laplacian acting on degree $k$ forms, that commutes with the maximal and/or minimal extension of $d_k$, in such a way that the arising complex is not trivial.  For instance, a naive approach would be to choose $\Delta_{k,\max}$, the $L^2$-maximal extension of $\Delta_k$, but this doesn't work because $d_{k,\max}$ does not commute with $\Delta_{k,\max}$. Let us give now more details by describing the structure of the paper.  First of all we  rework the problem in the suitable setting of
Hilbert complexes. Therefore in Section \ref{s1}, we introduce the necessary notation and background. Sections \ref{s2} and \ref{s3} are devoted to the construction of the complex of harmonic vectors and relative harmonic vectors, while in Sections \ref{s4} and \ref{s5} we proceed as far as possible with the analysis of the cohomology of these complexes. In Sections \ref{s6} and \ref{s7}, we prove the analogous result of Cappel, Deturck, Gluch and Miller in the  abstract framework of Fredholm complexes and for both the complexes of absolute and relative harmonic vectors, see Th. \ref{TH} and Th. \ref{TH-b}. These are two of the main results of this paper. More precisely, in the absolute case we have:
\begin{theo}
\label{THint}
Let $\CS_{\bu,M}$ and $\CS_{\bu,m}$ be two Hilbert complexes as in \eqref{M}-\eqref{m} with  $\CS_{\bu,M}$ a Fredholm complex. Let $H^{q}(\CS'_{\bu})$ be the cohomology  of the complex of  harmonic vectors of the pair $(\CS_{\bu,M},\CS_{\bu,m})$, see \eqref{homology}. Then we have an isomorphism of vector spaces
$$H^{q}(\CS'_{\bu})\oplus \mathcal{H}^{q-1}_{m}(\CS_{\bu,M},\CS_{\bu,m})\cong H^{q}(\CS_{\bu,M})\oplus H^{q-1}(\CS_{\bu,M}).$$ 
\end{theo}
In the relative case we have:
\begin{theo}
\label{TH-b-int}
Let $\CS_{\bu,M}$ and $\CS_{\bu,m}$ be two Fredholm complexes as in \eqref{M}-\eqref{m}. Let $H^{q}(\CS''_{\bu})$ be the cohomology  groups defined in \eqref{homology-b}. Then we have an isomorphism of vector spaces
$$H^{q}(\CS''_{\bu})\oplus \mathcal{H}^{q-1}_{m}(\CS_{\bu,M},\CS_{\bu,m})\cong H^{q}(\CS_{\bu,m})\oplus H^{q-1}(\CS_{\bu,M}).$$ 
\end{theo}
Finally, we complete this note with some geometric applications in Section \ref{s8}. First, we recall the definition of $L^2$ cohomology of a Riemannian manifold, in Section \ref{s8.1}. Then, in Section \ref{s8.2}, we show that our construction yields the same cohomology groups as in the work of Cappel, Deturck, Gluch and Miller. Finally, in Section \ref{s8.3}, we give some applications to smoothly Thom-Mather stratified spaces endowed with a $\hat{c}$-iterated  edge metric.\\

\noindent \textbf{Acknowledgments:} The first-named author was partially supported by 2024 Sapienza research grant ``New research trends in Mathematics at Castelnuovo'' and by INdAM - GNSAGA Project, codice CUP E53C24001950001.

%
%

\section{Setting and notations}
\label{s1}
We start with the definition and the basic properties of a Hilbert complex. We refer to \cite{BL} for an in-depth treatment.\\
A Hilbert complex $\CS_\bu=(\HF_\bu, d_\bu)$ is a  (bounded) complex of Hilbert spaces and closed operators:
\[
\xymatrix{
\CS_\bu:& \DS(d_0)\ar[r]^{d_0}&\DS(d_1)\ar[r]^{d_1}&\dots\ar[r]^{d_q}&\dots \ar[r]^{d_{n-1}}&\HH_n,
}
\]
where  
\[
d_q:\DS(d_q)\subset \HF_q\to \HF_{q+1},
\]
is a densely defined closed linear operator with domain $\DS(d_q)$  contained in some  Hilbert spaces $\HF_q$, for each $q=0,...,n-1$, and, in addition,
\[
d_q(\DS(d_q))\subseteq \DS(d_{q+1})\quad \mathrm{and}\quad d_q\circ d_{q-1}=0.
\]

We write
\[
H^q(\CS_\bu)=\frac{\ker (d_q:\DS(d_q)\to \HF_{q+1})}{\Im (d_{q-1}:\DS(d_{q-1})\to \HF_q)},
\]
for the cohomology of $\CS_\bu$. We write
\[
\bar H^q(\CS_\bu)=\frac{\ker (d_q:\DS(d_q)\to \HF_{q+1})}{\overline{\Im (d_{q-1}:\DS(d_{q-1})\to \HF_q)}},
\]
for the reduced cohomology of $\CS_\bu$.

The complex $\CS_\bu$ induces in a natural way a second Hilbert complex, the so-called  dual complex:
\[
\xymatrix{
\CS^*_\bu:& \DS(\de_n)\ar[r]^{\de_n}&\DS(\de_{n-1})\ar[r]^{\de_{n-1}}&\dots\ar[r]^{\de_q}&\dots \ar[r]^{\de_{0}}&\HF_0,
}
\]
where
\[
\de_q=d_{q-1}^*:\HF_{q}\to \HF_{q-1},
\]
denotes the adjoint operator of $d_{q-1}:\HF_{q-1}\to \HF_{q}$. Note that $\de_q:\HF_{q}\to \HF_{q-1}$ is a closed and densely defined operator as $d_{q-1}:\HF_{q-1}\to \HF_{q}$ is so.\\
For these complexes we have the so-called {\em weak Hodge decomposition} 
\beq
\label{weak}
\HF^q=\left(\ker (d_q)\cap \ker (\de_q)\right)\oplus \overline{\Im (d_{q-1})} \oplus \overline{\Im (\de_{q+1})},
\eeq
and consequently the isomorphism
\beq
\label{weak1}
 \bar H^q(\CS_\bu)\cong \ker (d_q)\cap \ker (\de_q).
\eeq

The Hodge-Laplace operator associated to the complex $\CS_\bu$ is the  densely defined, closed and self adjoint operator
\[
\Delta_q=\de_{q+1}\circ d_q+d_{q-1}\circ \de_q,
\]
with domain $\DS(\Delta_q)=\{s\in \DS(d_q)\cap\DS(\de_q): d_qs\in \DS(\de_{q+1}),\ \de_qs\in \DS(d_{q-1})\}$. Note that

$$
\ker (\Delta_q)=\ker (d_q)\cap \ker (\de_{q})\quad\quad \overline{\Im(\Delta_q)}= \overline{\Im (d_{q-1})} \oplus \overline{\Im (\de_{q+1})}.
$$

The complex is said to be {\em weak-Fredholm} when $\bar H^q(\CS_{\bu})$ has finite dimension for each $q=0,...,n$. 
The complex is said to be {\em Fredholm} when $H^q(\CS_{\bu})$ has finite dimension for each $q=0,...,n$. If this is the case then  $\Im (d_{q})$ and $\Im(\de_q)$ are both closed,  the following isomorphism holds true
\[
H^q(\CS_\bu)
\cong\ker (d_q)\cap \ker (\de_{q}),
\]
and the strong Hodge decomposition
\[
\HF^q=\left(\ker (d_q)\cap \ker (\de_{q})\right)\oplus \Im (d_{q-1}) \oplus \Im (\de_{q+1})
\]
holds true, as well.

\section{The complex of harmonic vectors}
\label{s2}

Let  $\CS_{\bu,M}=(\HF_\bu, d_{\bu,M})$ and $\CS_{\bu,m}=(\HF_\bu, d_{\bu,m})$

\begin{equation}
\label{M}
\xymatrix{
\CS_{\bu,M}:& \DS(d_{0,M})\ar[r]^{d_{0,M}}&\dots\ar[r]^{d_{q,M}}&\dots \ar[r]^{d_{n-1,M}}&\HF_n,
}
\end{equation}

\begin{equation}
\label{m}
\xymatrix{
\CS_{\bu,m}:& \DS(d_{0,m})\ar[r]^{d_{0,m}}&\dots\ar[r]^{d_{q,m}}&\dots \ar[r]^{d_{n-1,m}}&\HF_n,
}
\end{equation}
be two given Hilbert complexes such that for each $q=0,...,n-1$, $d_{q,M}$ is a closed extension of $d_{q,m}$, that is $\DS(d_{q,m})\subset \DS(d_{q,M})$ and $d_{q,M}|_{\DS(d_{q,m})}=d_{q,m}$. It is clear that $d^*_{q,m}:\HF_{q+1}\rightarrow \HF_q$, that is the adjoint of $d_{q,m}:\HF_q\rightarrow \HF_{q+1}$, is a closed extension of $d_{q,M}^*:\HF_{q+1}\rightarrow \HF_q$, the adjoint of $d_{q,M}:\HF_q\rightarrow \HF_{q+1}$. Let us denote $\de_{q,m}:=d^*_{q-1,M}$, $\de_{q,M}:=d^*_{q-1,m}$ and let us define for each $q=0,...,n$ the operator $$P_q:\HF_q\rightarrow \HF_q\quad\quad P_q:=d_{q-1,M}\circ\de_{q,M}+\de_{q+1,M}\circ d_{q,M}$$ with domain $$\DS(P_q)=\{s\in \DS(d_{q,M})\cap\DS(\de_{q,M}): d_{q,M}s\in \DS(\de_{q+1,M}),\ \de_{q,M}s\in \DS(d_{q-1,M})\}.$$ Note that the operator $P_q$ can be thought as a (possibly non-closed) extension of both $\Delta_{q,M}$ and $\Delta_{q,m}$, the Hodge-Laplacians of the complexes $\CS_{\bu,M}$ and $\CS_{\bu,m}$, respectively. In particular we have $\ker(\Delta_{q,m})\subset \ker(P_q)$ and $\ker(\Delta_{q,M})\subset \ker(P_q)$. Our goal now is to show how $\ker(P_q)$ is related to the cohomology of $\CS_{\bu,M}$ and $\CS_{\bu,m}$.

\begin{lem}\label{l0} For each $q$, if $u\in \ker (P_q)$, then $d_{q,M} u\in \ker (P_{q+1})$.
\end{lem}
\begin{proof}  Since $u\in \ker (P_q)$, it follows that
\[
\de_{q+1,M} (d_{q,M} u)=-d_{q-1,M}( \de_{q,M} u).
\]
Consequently,  $\de_{q+1,M} (d_{q,M} u)\in \DS(d_{q,M})$ which in turn implies that $d_{q,M}u\in \DS(P_{q+1})$. Finally, note that $$P_{q+1}(d_{q,M}u)=d_{q,M}(\de_{q+1,M} (d_{q,M} u))=-d_{q,M}(d_{q-1,M}( \de_{q,M} u))=0$$ as required.
\end{proof}

Dually we have:

\begin{lem}\label{l0.1} For each $q$, if $u\in \ker (P_q)$, then $\de_{q,M} u\in \ker (P_{q-1})$.
\end{lem}
\begin{proof}  Since $u\in \ker (P_q)$, it follows that
\[
d_{q-1,M}(\de_{q,M}u)=-\de_{q+1,M}(d_{q,M}u)
\]
and so we can conclude that $d_{q-1,M}(\de_{q,M}u)\in \DS(\de_{q,M})$ and thus  $\de_{q,M} u\in \DS( P_{q-1})$. We conclude by verifying that 
\[
P_{q-1} (\de_{q,M} u)= \de_{q,M}( d_{q-1,M}( \de_{q,M} u))=-\de_{q,M}( \de_{q+1,M}( d_{q,M}  u)) =0.
\]
\end{proof}

%
%
%
In view of Lemma \ref{l0} we have the next
\begin{corol} We can form the following complex  of vector spaces
\[
\xymatrix{
\CS'_{\bu}:& \ker(P_0)\ar[r]^{d'_{0}}&\dots \ar[r]&\ker (P_q)\ar[r]^{d'_{q}}&\ker(P_{q+1})\ar[r]& \dots \ar[r]^{d'_{n-1}}&\ker (P_n),
}
\]
where $d'_q$ is the restriction of $d_{q,M}$ to $\ker(P^q)$. 
\end{corol}
We refer to the above complex as the {\em complex of harmonic vectors} of the pair $(\CS_{\bu,M},\CS_{\bu,m})$.

\begin{defi} We call the cohomology of $\CS'_{\bu}$:
\begin{equation}
\label{homology}
H^q(\CS'_{\bu})=\frac{\ker (d'_q:\ker P_q\to \ker P_{q+1})}{\Im (d'_{q-1}:\ker P_{q-1}\to \ker P_{q})}
\end{equation}
 the cohomology of the complex of harmonic vectors of the pair $(\CS_{\bu,M},\CS_{\bu,m})$.
\end{defi}


\section{The complex of relative harmonic vectors}
\label{s3}
In this section we construct another complex from $\ker(P_q)$. We follows the lines of the previous section but we replace $\CS_{\bu,M}$ with $\CS_{\bu,m}$.

\begin{lem}\label{l0-b} For each $q$, if $u\in \ker (P_q)\cap \DS(d_{q,m})$, then $d_{q,m} u\in \ker (P_{q+1})\cap \DS(d_{q+1,m})$.
\end{lem}
\begin{proof}  Since $u\in \ker (P_q)$, it follows that
\[
\de_{q+1,M} (d_{q,M} u)=-d_{q-1,M}( \de_{q,M} u).
\]

Consequently,  $\de_{q+1,M} (d_{q,m} u)\in \DS(d_{q,M})$ which in turn implies that $d_{q,m}u\in \DS(P_{q+1})$. Finally note that $$P_{q+1}(d_{q,m}u)=d_{q,M}(\de_{q+1,M} (d_{q,m} u))=-d_{q,M}(d_{q-1,M}( \de_{q,M} u))=0$$ as required.
\end{proof}


\begin{corol} We can form the following complex  of vector spaces
\[
\xymatrix{
\CS''_{\bu}:& \dots \ar[r]&\ker (P_q)\cap \DS(d_{q,m})\ar[r]^{d''_{q}}&\ker(P_{q+1})\cap \DS(d_{q+1,m})\ar[r]& \dots 
}
\]
where $\DS(d''_q)=\ker P_q\cap \DS(d_{q,m})$. 
\end{corol}

Analogously to the previous section we refer to the above complex as the {\em complex of relative harmonic vectors} of the pair $(\CS_{\bu,M},\CS_{\bu,m})$


\begin{defi} We call the cohomology of $\CS''_{\bu}$:
\begin{equation}
\label{homology-b}
H^q(\CS''_{\bu})=\frac{\ker (d''_q:\ker (P_q)\cap \DS(d_{q,m})\to \ker (P_{q+1})\cap \DS(d_{q+1,m}))}{\Im (d''_{q-1}:\ker (P_{q-1})\cap \DS(d_{q-1,m})\to \ker (P_q)\cap \DS(d_{q,m}))}
\end{equation}
 the cohomology of the complex of relative harmonic vectors of the pair $(\CS_{\bu,M},\CS_{\bu,m})$.
\end{defi}


\section{Harmonic cohomology decomposition for general Hilbert complexes}
\label{s4}

In this section we give the first structure results concerning the  cohomology of the complex of harmonic vectors of the pair  $(\CS_{\bu,M},\CS_{\bu,m})$. Consider the vector spaces defined in \eqref{homology}
\[
H^q(\CS'_{\bu})=\frac{\ker (d'_q:\ker (P_q)\to \ker (P_{q+1}))}{\Im (d'_{q-1}:\ker (P_{q-1})\to \ker (P_{q}))}
\]
and let
\beq\label{b1}
\HF^q=\left(\ker (d_{q,M})\cap \ker (\de_{q,m})\right)\oplus \overline{\Im (d_{{q-1},M})} \oplus \overline{\Im (\de_{q+1,m})}
\eeq

\noindent be the Hodge decomposition induced by the  complex $(\DS(d_{\bu,M}),d_{\bu,M})$, see  \eqref{weak}. Let $\omega\in \ker(d'_q)$ and let $\omega=\omega_1+\omega_2+\omega_3$ be the decomposition of $\omega$ w.r.t. \eqref{b1}, with $\omega_1\in \ker (d_{q,M})\cap \ker( \de_{q,m})$, $\omega_2\in \overline{\Im (d_{{q-1},M})}$ and  $\omega_3\in  \overline{\Im (\de_{q+1,m})}$. First note that $\ker(d'_q)=\ker(P_q)\cap \ker(d_{q,M})$ and that $\omega_3=0$, as $\omega\in \ker(P_q)\cap \ker(d_{q,M})$ and $\ker(d_{q,M})$  is orthogonal to $\overline{\Im(\de_{q+1,m})}$. Note in addition that $$\left(\ker (d_{q,M})\cap \ker (\de_{q,m})\right)\subset \left(\ker(P_q)\cap \ker(d_{q,M})\right)$$ and thus $$\omega_2\in \ker(P_q)\cap \ker(d_{q,M})\cap \overline{\Im(d_{q-1,M})}=\ker(P_q)\cap\overline{\Im(d_{q-1,M})}.$$ Therefore, so far, we showed that $$\ker(P_q)\cap \ker(d_{q,M})\subset (\ker (d_{q,M})\cap \ker (\de_{q,m}))\oplus \left(\overline{\Im(d_{q-1,M})}\cap\ker(P_q)\right).$$ On the other hand it is clear that if $\omega_1$ and $\omega_2$ are arbitrarily fixed forms lying in $\ker (d_{q,M})\cap \ker (\de_{q,m})$ and $\overline{\Im(d_{q-1,M})}\cap\ker(P_q)$, respectively, then $\omega_1+\omega_2\in \ker(P_q)\cap \ker(d_{q,M})$ and so we can conclude that $$\ker(P_q)\cap \ker(d_{q,M})\supset (\ker d_{q,M}\cap \ker \de_{q,m})\oplus \left(\overline{\Im(d_{q-1,M})}\cap\ker(P_q)\right).$$

Summarising we proved that $$\ker(P_q)\cap \ker(d_{q,M})= \left(\ker( d_{q,M})\cap \ker( \de_{q,m})\right)\oplus \left(\overline{\Im(d_{q-1,M})}\cap\ker(P_q)\right)$$ and consequently, by applying the weak Hodge decomposition \eqref{weak}, we get  $$\frac{\ker(P_q)\cap \ker(d_{q,M})}{\Im(d'_{q-1})}\cong (\ker (d_{q,M})\cap \ker (\de_{q,m}))\oplus \frac{\overline{\Im(d_{q-1,M})}\cap\ker(P_q)}{\Im(d'_{q-1})}$$ as $\Im(d'_q)\subset \Im(d_{q,M})$ and $\Im(d_{q,M})$ is orthogonal to $\ker(\de_{q,m})$.  We can summarize the above discussion with the next 

\begin{prop}
\label{firststep}
Let $H^q(\CS'_{\bu})$ be the vector space defined in \eqref{homology}. Then, we have: $$H^q(\CS'_{\bu})\cong (\ker (d_{q,M})\cap \ker( \de_{q,m}))\oplus \frac{\overline{\Im(d_{q-1,M})}\cap\ker(P_q)}{\Im(d'_{q-1})}.$$
\end{prop}

Observe now that $\ker (P_q)\cap \overline{\Im(d_{q-1,M})}$ is contained in the domain of $\de_{q,M}$, since the domain of $P_q$ is so. Therefore we can introduce the following map
\[
\de'_{q,M}:\ker(P_q)\cap 
\overline{\Im (d_{q-1,M})} \to \HF_{q-1},
\]
which is simply defined as the restriction of $\de_{q,M}$ on $\ker(P_q)\cap \overline{\Im(d_{q-1,M})}$. We have the following 

\begin{prop}
\label{morphism}
The map $\de'_{q,M}:\ker (P_q)\cap 
\overline{\Im (d_{q-1,M})} \to \HF_{q-1}$ induces a morphism 
$$\hat \de_{q}:\frac{\ker (P_q)\cap 
\overline{\Im (d_{q-1,M})}}{\Im(d'_{q-1})}
\to \frac{\ker(d_{q-1,M})}{\Im(d_{q-2,M})}.$$
\end{prop}

\begin{proof}
Let $u\in \ker(P_q)\cap
\overline{\Im (d_{q-1,M})} $. Since $u\in  \ker (P_q)\cap \ker(d_{q,M})$, we have
\[
0=P_q u=d_{q-1,M}(\de_{q,M}u),
\]
and hence, $\de_{q,M}u \in \ker( d_{q-1,M})$. Assume now that  $u\in \Im (d'_{q-1})$, then $u=d_{q-1,M} v$, for some $v\in \ker(P_{q-1})$. Since $P_{q-1} v=0$, we have that $d_{q-2,M}(\de_{q-1,M} v)=-\de_{q,M}( d_{q-1,M} v)$, and therefore $\de_{q,M}u=\de_{q,M}(d_{q-1,M} v)=-d_{q-2,M}(\de_{q-1,M}v)$ which allows us to conclude that $\de_{q,M}u\in \Im (d_{q-2,M}$).
\end{proof}

We conclude this section with the following lemma that we believe having a per se  interest because it shows a Hodge decomposition that to the best of our knowledge is new.

\begin{lem}
\label{new-Hodge-weak}
Let $\CS_{\bu,M}$ be a weak Fredholm complex. Then, the following {Hodge decomposition} holds true:
$$\HF_q=\left(\ker(d_{q,m})\cap\ker(\de_{q,m})\right)\oplus \left(\overline{\Im(d_{q-1,M})}+\overline{\Im(\de_{q+1,M})}\right).$$
\end{lem}

\begin{proof}
Let us denote with $X$ the orthogonal complement of $\overline{\Im(d_{q-1,M})}\oplus\overline{\Im(\de_{q+1,m})}$. Since $\CS_{\bu,M}$ is a weak-Fredholm complex, we know that  $\overline{\Im(d_{q-1,M})}\oplus\overline{\Im(\de_{q+1,m})}$ is a closed subspace of $\HF_q$ and $X$ has finite dimension. Moreover, $\overline{\Im(d_{q-1,M})}\oplus\overline{\Im(\de_{q+1,m})}\subset \overline{\Im(d_{q-1,M})}+\overline{\Im(\de_{q+1,M})}$. This tells us that we have the following orthogonal decomposition for $\overline{\Im(d_{q-1,M})}+\overline{\Im(\de_{q+1,M})}$: 
\begin{multline}
\nonumber \overline{\Im(d_{q-1,M})}+\overline{\Im(\de_{q+1,M})}=\\
\left(\overline{\Im(d_{q-1,M})}\oplus\overline{\Im(\de_{q+1,m})}\right)\oplus \left(X\cap \left(\overline{\Im(d_{q-1,M})}+\overline{\Im(\de_{q+1,M})}\right)\right).
\end{multline}
Note that $X \cap \left(\overline{\Im(d_{q-1,M})}+\overline{\Im(\de_{q+1,M})}\right)$ is a finite dimensional vector space because $X$ is so. Therefore, $X \cap \left(\overline{\Im(d_{q-1,M})}+\overline{\Im(\de_{q+1,M})}\right)$ is closed in $\HF_q$ and so we can conclude that $\overline{\Im(d_{q-1,M})}+\overline{\Im(\de_{q+1,M})}$ is also closed in $\HF_q$ because it is the sum of two mutually orthogonal closed subspace of $\HF_q$. Finally, since it is  immediate to verify that 
\[
\left(\ker(d_{q,m})\cap\ker(\de_{q,m})\right)^{\bot}=\overline{\overline{\Im(d_{q-1,M})}+\overline{\Im(\de_{q+1,M})}}=\overline{\Im(d_{q-1,M})}+\overline{\Im(\de_{q+1,M})},
\] 
and  analogously 
\[
\left(\overline{\Im(d_{q-1,M})}+\overline{\Im(\de_{q+1,M})}\right)^{\bot}=\ker(d_{q,m})\cap\ker(\de_{q,m}),
\]
we can conclude that the stated orthogonal decomposition  holds true. 
\end{proof}

\section{Relative harmonic cohomology decomposition for general Hilbert complexes}
\label{s5}

Analogously to the previous section we start to describe the structure of the cohomology of the complex of relative harmonic vectors. Consider the vector space defined in \eqref{homology-b}
\[
H^q(\CS''_{\bu})=\frac{\ker (d''_q:\ker (P_q)\cap \DS(d_{q,m})\to \ker (P_{q+1})\cap \DS(d_{q+1,m}))}{\Im (d''_{q-1}:\ker (P_{q-1})\cap \DS(d_{q-1,m})\to \ker (P_q)\cap \DS(d_{q,m}))}
\]

\noindent and let

\beq\label{b1-b}
\HF^q=\left(\ker (d_{q,m})\cap \ker (\de_{q,M})\right)\oplus \overline{\Im (d_{{q-1},m})} \oplus \overline{\Im (\de_{q+1,M})}
\eeq
be the Hodge decomposition induced by the  complex $(\DS(d_{\bu,m}),d_{\bu,m})$, see \eqref{weak}. Let $\omega\in \ker(d''_q)$ and let $\omega=\omega_1+\omega_2+\omega_3$ be the decomposition of $\omega$ w.r.t. \eqref{b1-b}, with $\omega_1\in \ker (d_{q,m})\cap \ker( \de_{q,M})$, $\omega_2\in \overline{\Im (d_{{q-1},m})}$ and  $\omega_3\in  \overline{\Im (\de_{q+1,M})}$. 
First note that $\ker(d''_q)=\ker(P_q)\cap \ker(d_{q,m})$ and that $\omega_3=0$, since $\omega\in \ker(P_q)\cap \ker(d_{q,m})$ and $\ker(d_{q,m})$  is orthogonal to $\overline{\Im(\de_{q+1,M})}$. 
Note in addition that $$\left(\ker (d_{q,m})\cap \ker (\de_{q,M})\right)\subset \left(\ker(P_q)\cap \ker(d_{q,m})\right)$$ and thus 
$$\omega_2\in \ker(P_q)\cap \ker(d_{q,m})\cap \overline{\Im(d_{q-1,m})}=\ker(P_q)\cap\overline{\Im(d_{q-1,m})}.$$ 
Therefore, we can deduce that 
$$\ker(P_q)\cap \ker(d_{q,m})\subset (\ker (d_{q,m})\cap \ker (\de_{q,M}))\oplus \left(\overline{\Im(d_{q-1,m})}\cap\ker(P_q)\right).$$ 
Conversely, if $\omega_1$ and $\omega_2$ are arbitrarily fixed forms lying in $\ker (d_{q,m})\cap \ker (\de_{q,M})$ and $\overline{\Im(d_{q-1,m})}\cap\ker(P_q)$, respectively, then $\omega_1+\omega_2\in \ker(P_q)\cap \ker(d_{q,m})$ and so we can conclude that $$\ker(P_q)\cap \ker(d_{q,m})\supset (\ker d_{q,m}\cap \ker \de_{q,M})\oplus \left(\overline{\Im(d_{q-1,m})}\cap\ker(P_q)\right).$$
So far we proved that 
$$\ker(P_q)\cap \ker(d_{q,m})= \left(\ker( d_{q,m})\cap \ker( \de_{q,M})\right)\oplus \left(\overline{\Im(d_{q-1,m})}\cap\ker(P_q)\right)$$ and consequently, by applying the weak Hodge decomposition \eqref{weak}, we get  
$$\frac{\ker(P_q)\cap \ker(d_{q,m})}{\Im(d''_{q-1})}\cong (\ker (d_{q,m})\cap \ker (\de_{q,M}))\oplus \frac{\overline{\Im(d_{q-1,m})}\cap\ker(P_q)}{\Im(d''_{q-1})}$$ since $\Im(d''_q)\subset \Im(d_{q,m})$ and $\Im(d_{q,m})$ is orthogonal to $\ker(\de_{q,M})$. Summarizing we have

\begin{prop}
\label{firststep-b}
Let $H^q(\CS''_{\bu})$ be the vector space defined in \eqref{homology}. Then, we have: $$H^q(\CS''_{\bu})\cong (\ker (d_{q,m})\cap \ker (\de_{q,M}))\oplus \frac{\overline{\Im(d_{q-1,m})}\cap\ker(P_q)}{\Im(d''_{q-1})}.$$
\end{prop}

Consider now the following map
\[
\de''_{q,M}:\ker(P_q)\cap 
\overline{\Im (d_{q-1,m})} \to \HF_{q-1},
\]
which is simply defined as the restriction of $\de_{q,M}$ on $\ker(P_q)\cap \overline{\Im(d_{q-1,m})}$. We have the following 

\begin{prop}
\label{morphism-b}
The map $\de''_{q,M}:\ker (P_q)\cap 
\overline{\Im (d_{q-1,m})} \to \HF_{q-1}$ induces a morphism $$\check \de_q:\frac{\ker (P_q)\cap 
\overline{\Im (d_{q-1,m})}}{\Im(d''_{q-1})}
\to \frac{\ker(d_{q-1,M})}{\Im(d_{q-2,M})}.$$
\end{prop}

\begin{proof}
Let $u\in \ker(P_q)\cap\overline{\Im (d_{q-1,m})} $. Since $u\in  \ker (P_q)\cap \ker(d_{q,m})$, we have
\[
0=P_q u=d_{q-1,M}(\de_{q,M}u),
\]
and hence, $\de_{q,M}u \in \ker( d_{q-1,M})$. Assume now that  $u\in \Im (d''_{q-1})$, then $u=d_{q-1,m} v$, for some $v\in \ker(P_{q-1})$. Since $P_{q-1} v=0$, we have that $d_{q-2,M}(\de_{q-1,M} v)=-\de_{q,M}( d_{q-1,M} v)$, and therefore $\de_{q,M}u=\de_{q,M}(d_{q-1,M} v)=-d_{q-2,M}(\de_{q-1,M}v)$, which allows us to conclude that $\de_{q,M}u\in \Im (d_{q-2,M}$).
\end{proof}

\section{Harmonic cohomology decomposition for Fredholm complexes}
\label{s6}

This section contains one of the main result of this paper. We show that  Prop. \ref{firststep} can be significantly refined if assume that $\CS_{\bu,M}$ is  a Fredholm complex. So we make the following

\begin{ass} From now on we assume that $\CS_{\bu,M}$ is  a Fredholm complex.
\end{ass}

First, we start with the following refinement of Lemma \ref{new-Hodge-weak}.

\begin{lem}
\label{newHodge}
The following properties hold true:
\begin{enumerate}
\item We have the following Hodge decomposition $$\HF_q=\left(\ker(d_{q,m})\cap\ker(\de_{q,m})\right)\oplus \left(\Im(d_{q-1,M})+\Im(\de_{q+1,M})\right).$$ In particular $\Im(d_{q-1,M})+\Im(\de_{q+1,M})$ is a closed subspace of $\HF_q$.

\item The range of $P_q:\DS(P_q)\subset\HF_q\rightarrow \HF_q$ satisfies $$\Im(P_q)=\Im(d_{q-1,M})+\Im(\de_{q+1,M}).$$ In particular $P_q$ is a closed range operator. 
\end{enumerate}
\end{lem}

\begin{proof}
The first item above follows by adopting the same strategy used in the proof of Lemma \ref{new-Hodge-weak}. Indeed, let $X$ be the orthogonal complement of $\Im(d_{q-1,M})\oplus\Im(\de_{q+1,m})$. Since $\CS_{\bu,M}$ is a Fredholm complex, we know that  $\Im(d_{q-1,M})\oplus\Im(\de_{q+1,m})$ is a closed subspace of $\HF_q$, $X$ has finite dimension, and moreover $\Im(d_{q-1,M})\oplus\Im(\de_{q+1,m})\subset \Im(d_{q-1,M})+\Im(\de_{q+1,M})$. Therefore, we have the following orthogonal decomposition for $\Im(d_{q-1,M})+\Im(\de_{q+1,M})$: 
\begin{multline}
\nonumber \Im(d_{q-1,M})+\Im(\de_{q+1,M})=\\\left(\Im(d_{q-1,M})\oplus\Im(\de_{q+1,m})\right)\oplus \left(X\cap \left(\Im(d_{q-1,M})+\Im(\de_{q+1,M})\right)\right).\end{multline}
Note that 
$X \cap \left(\Im(d_{q-1,M})+\Im(\de_{q+1,M})\right)$ is a finite dimensional vector space because $X$ is so. Therefore, $X \cap \left(\Im(d_{q-1,M})+\Im(\de_{q+1,M})\right)$ is closed in $\HF_q$, and  we can conclude that $\Im(d_{q-1,M})+\Im(\de_{q+1,M})$ is also closed in $\HF_q$, because it is the sum of two mutually orthogonal closed subspace of $\HF_q$. Finally, given that $$\left(\ker(d_{q,m})\cap\ker(\de_{q,m})\right)^{\bot}=\overline{\Im(d_{q-1,M})+\Im(\de_{q+1,M})}=\Im(d_{q-1,M})+\Im(\de_{q+1,M})$$ and  analogously $$\left(\Im(d_{q-1,M})+\Im(\de_{q+1,M})\right)^{\bot}=\ker(d_{q,m})\cap\ker(\de_{q,m})$$ we can conclude that the stated orthogonal decomposition holds true. Let us deal know with the second part. It suffices to show that $\Im(P_q)\supset (\Im(d_{q-1,M})+\Im(\de_{q+1,M}))$, because the other inclusion, namely  $\Im(P_q)\subset (\Im(d_{q-1,M})+\Im(\de_{q+1,M}))$, it is clear. Let $\beta_1\in \Im(d_{q-1,M})$. Since $\CS_{\bu,M}$ is a Fredholm complex, we can decompose $\HF_{q-1}$ as $$\HF_{q-1}=(\ker(d_{q-1,M})\cap \ker(\de_{q-1,m}))\oplus \Im(d_{q-2,M})\oplus \Im(\de_{q,m}).$$ 

It is therefore clear that there exists $\beta_2\in \DS(d_{q-1,M})\cap \Im(\de_{q,m})$ such that $d_{q-1,M}\beta_2=\beta_1$.  By applying the Hodge decomposition to $\HF_q$, $$\HF_q=(\ker(d_{q,M})\cap \ker(\de_{q,m}))\oplus \Im(d_{q-1,M})\oplus \Im(\de_{q+1,m})$$ we obtain the existence of an element $\beta_3\in \Im(d_{q-1,M})\cap \DS(\de_{q,m})$ such that $\de_{q,m}\beta_3=\beta_2$. Let $\beta_4\in \DS(d_{q-1,M})$ be such that $d_{q-1,M}\beta_4=\beta_3$. Summarizing we have $$\beta_1=d_{q-1,M}(\de_{q,m}(d_{q-1,M}\beta_4))=P_q(d_{q-1,M}\beta_4)$$ and so we can conclude that $\Im(d_{q-1,M})\subset \Im(P_q)$. A completely analogous argument proves that $\Im(\de_{q+1,M})\subset \Im(P_q)$. We can thus conclude that $$\Im(P_q)\supset (\Im(d_{q-1,M})+\Im(\de_{q+1,M}))$$ and eventually that $$\Im(P_q)= (\Im(d_{q-1,M})+\Im(\de_{q+1,M}))$$ as required.
\end{proof}

Let us now consider the vector space $$\frac{\ker(P_q)\cap \Im(d_{q-1,M})}{\Im(d'_{q-1})}\oplus \mathcal{H}^{q-1}_{m}(\CS_{\bu,M},\CS_{\bu,m})$$
with $$\mathcal{H}^{q-1}_{m}(\CS_{\bu,M},\CS_{\bu,m}):=\ker(d_{q-1,m})\cap \ker(\de_{q-1,m}).$$ 

Note that $$\mathcal{H}^{q-1}_{m}(\CS_{\bu,M},\CS_{\bu,m})\subset \ker(d_{q-1,M})\cap\ker(\de_{q-1,m})$$ and, since $(\CS_{q,M})$ is a Fredholm complex, we get an inclusion $$\mathcal{H}^{q-1}_{m}(\CS_{\bu,M},\CS_{\bu,m})\hookrightarrow H^{q-1}(\CS_{\bu,M})$$ given by $$v\mapsto [v]$$ with $[v]$ the class induced by $v$ in $H^q(\CS_{\bu,M})$. Let us define now $$\tilde{\de}_q:\frac{\ker(P_q)\cap \Im(d_{q-1,M})}{\Im(d'_{q-1})}\oplus \mathcal{H}^{q-1}_{m}(\CS_{\bu,M},\CS_{\bu,m})\rightarrow H^{q-1}(\CS_{\bu,M})$$ as  the morphism given by  
\begin{equation}
\label{def}
\tilde{\de_q}([u]\oplus v):=\hat{\de}_{q}([u])+[v]
\end{equation}
 for any $[u]\in \frac{\ker(P_q)\cap \Im(d_{q-1,M})}{\Im(d'_{q-1})}$ and $v\in \mathcal{H}^{q-1}_{m}(\CS_{\bu,M},\CS_{\bu,m})$, and where $\hat{\de}_q$ is the morphism introduced in Prop. \ref{morphism}. Note that by Prop. \ref{morphism} $\hat{\de}_{q}([u])$ is a cohomology class in $H^{q-1}(\CS_{\bu,M})$ and, by the above discussion, $[v]$ is also a  cohomology class in $H^{q-1}(\CS_{\bu,M})$. Hence on the right-hand side of \eqref{def} we have the sum of two cohomology classes in $H^{q-1}(\CS_{\bu,M})$. The next theorem is the crucial step toward the main result of this section. 

\begin{theo}
\label{crucial}
Let $\CS_{\bu,M}$ and $\CS_{\bu,m}$ be two Hilbert complexes as in \eqref{M}-\eqref{m} with  $\CS_{\bu,M}$ a Fredholm complex. Then the morphism of vector spaces:
$$\tilde{\de}_q:\frac{\ker(P_q)\cap \Im(d_{q-1,M})}{\Im(d'_{q-1})}\oplus \mathcal{H}^{q-1}_{m}(\CS_{\bu,M},\CS_{\bu,m})\rightarrow H^{q-1}(\CS_{\bu,M})$$ defined above is an isomorphism.
\end{theo}
\begin{proof}
Let $w\in \ker(P_q)\cap \Im(d_{q-1,M})$, $v\in \mathcal{H}^{q-1}_{m}(\CS_{\bu,M},\CS_{\bu,m})$ and $z\in \DS(d_{q-2,M})$ such that $\delta_{q,M}w+v=d_{q-2,M}z$. Thanks to the Hodge decomposition of Lemma \ref{newHodge} we get that $v=0$ and $\delta_{q,M}w=d_{q-2,M}z$. Therefore $\tilde{\de}_{q}$ is injective if and only if $$\hat \de_q:\frac{\ker (P_q)\cap 
\Im (d_{q-1,M})}{\Im(d'_{q-1})}
\to \frac{\ker(d_{q-1,M})}{\Im(d_{q-2,M})}$$ is injective. Let us show now that $\hat{\de}_{q}$ is actually injective.  Take $a\in \ker (P_q)\cap
\Im (d_{q-1,M})$ and assume that $\de_{q,M} a\in \Im( d_{q-2,M})$. Then, we have to show that $a\in \Im (d'_{q-1})$. Since $a\in \Im (d_{q-1,M})$, we may write $a=d_{q-1,M}u$, with $u\in \DS (d_{q-1,M})\cap \Im (\de_{q,m})$ thanks to the Hodge decomposition 
$$
\HF_{q-1}=\left(\ker (d_{q-1,M})\cap \ker( \de_{q-1,m})\right)\oplus \Im( d_{q-2,M}) \oplus \Im (\de_{q,m}).
$$ 
Moreover, by applying again the Hodge decomposition in degree $q$ and arguing in the same manner, we find an element $v\in \DS(\de_{q,m})\cap \Im(d_{q-1,M})$ such that $\delta_{q,m}v=u$ and $d_{q-1,M}(\delta_{q,m}v)=a.$ Note that $u=\delta_{q,m}v\in \DS(P_{q-1})$ and $P_{q-1}u=\delta_{q,M}a$.

By assumption we know that  $d_{q-2,M}c=\de_{q,M}a$   whence  $P_{q-1} u=d_{q-2,M}c$, for some $c\in \DS(d_{q-2,M})$. Using the Hodge decomposition
\[
\HF_{q-2}=\left(\ker (d_{q-2,M})\cap \ker (\de_{q-2,m})\right)\oplus \Im( d_{ q-3,M}) \oplus \Im (\de_{q-1,m}),
\]
 we can assume without loss of generality that   
\[
c\in \DS(\d_{q-2,M})\cap \Im \de_{q-1,m}.
\]

Therefore, $c=\de_{q-1,m}x$, for some $x\in\DS(\de_{q-1,m})$. Repeating the same argument, we can find $y\in \DS(\d_{q-2,M})$, such that
\[
c=\de_{q-1,m}x=\de_{q-1,m}(\d_{q-2,M}y).
\]

Since the domain of the minimal extension is contained in the domain of the maximal extension, we have
\[
P_{q-1} u=d_{q-2,M}c=d_{q-2,M}(\de_{q-1,m}(\d_{q-2,M}y))=P_{q-1}(\d_{q-2,M}y).
\]

Set $z=u-\d_{q-2,M}y$. Then, on one side
\[
P_{q-1} z=0,
\]
and on the other
\[
a=d_{q-1,M}u=d_{q-1,M}(u- d_{q-2,M}y)=d_{q-1,M}z,
\]
i.e. $a\in \Im(d'_{q-1})$, as desired. Now we show that $\tilde{\de}_q$ is surjective. Let $u\in \ker(d_{q-1,M})$. Thanks to the Hodge decomposition of Lemma \ref{newHodge}, we can write $u=u_0+d_{q-2,M}u_1+\de_{q,M}u_2$ with $u_0\in \mathcal{H}^{q-1}_{\min}(\CS_{\bu,M},\CS_{\bu,m})$. Moreover, by applying the Hodge decomposition to $\HF_q$: $$\HF_{q}=\left(\ker (d_{q,M})\cap \ker( \de_{q,m})\right)\oplus \Im( d_{q-1,M}) \oplus \Im (\de_{q+1,m})
$$ we obtain that $u_2=d_{q-1,M}\gamma$. Note that $d_{q-1,M}\gamma\in \DS(P_q)$ and $P_q(d_{q-1,M}\gamma)=0$. Indeed, it is obvious that $d_{q-1,M}\gamma\in \ker(d_{q,M})$. Furthermore, $d_{q-1,M}\gamma=u_2\in \DS(\de_{q,M})$ and  $\de_{q,M}(u_2)=u-u_0-d_{q-2,M}u_1\in \ker(d_{q-1,M})$. Therefore, $d_{q-1,M}\gamma\in \DS(P_q)$ and $P_q(d_{q-1,M}\gamma)=d_{q-1,M}(u-u_0-d_{q-2,M}u_1)=0$. Summarizing, we have $$d_{q-1,M}\gamma\in \ker(P_{q})\cap \Im(d_{q-1,M})\quad \mathrm{and}\quad \hat{\de}_q([d_{q-1,M}\gamma])=[u-u_0]\in H^{q-1}(\CS_{\bu,M}).$$ We can thus conclude that $$\tilde{\de_q}([d_{q-1,M}\gamma]\oplus u_0)=[u]\in H^{q-1}(\CS_{\bu,M})$$ and eventually that $$\tilde{\de}_q:\frac{\ker(P_q)\cap \Im(d_{q-1,M})}{\Im(d'_{q-1})}\oplus \mathcal{H}^q_{m}(\CS_{\bu,M},\CS_{\bu,m})\rightarrow H^{q-1}(\CS_{\bu,M})$$ is an isomorphism of vector spaces.
\end{proof}

\begin{theo}
\label{TH}
Let $\CS_{\bu,M}$ and $\CS_{\bu,m}$ be two Hilbert complexes as in \eqref{M}-\eqref{m} with  $\CS_{\bu,M}$ a Fredholm complex. Let $H^{q}(\CS'_{\bu})$ be the cohomology  of the complex of  harmonic vectors of the pair $(\CS_{\bu,M},\CS_{\bu,m})$, see \eqref{homology}. Then, we have an isomorphism of vector spaces
$$H^{q}(\CS'_{\bu})\oplus \mathcal{H}^{q-1}_{m}(\CS_{\bu,M},\CS_{\bu,m})\cong H^{q}(\CS_{\bu,M})\oplus H^{q-1}(\CS_{\bu,M}).$$ 
\end{theo}

\begin{proof}
Thanks to Prop. \ref{firststep} and Th. \ref{crucial}, we have the following chain of isomorphisms:
$$
\begin{aligned}
&H^{q}(\CS'_{\bu})\oplus \mathcal{H}^{q-1}_{m}(\CS_{\bu,M},\CS_{\bu,m})\cong\\
&\cong (\ker (d_{q,M})\cap \ker( \de_{q,m}))\oplus \frac{\overline{\Im(d_{q-1,M})}\cap\ker(P_q)}{\Im(d'_{q-1})}\oplus \mathcal{H}^{q-1}_{m}(\CS_{\bu,M},\CS_{\bu,m})\\
&\cong H^q(\CS_{\bu,M})\oplus H^{q-1}(\CS_{\bu,M}).
\end{aligned}
$$
\end{proof}

Looking at the above theorem, it is clear that $H^{q}(\CS'_{\bu})$ is not fully determined by the cohomology of the complexes $\CS_{\bu,M}$ and $\CS_{\bu,m}$. However, we have upper and lower bounds for $\dim(H^{q}(\CS'_{\bu}))$, that only depend  on the cohomology of $\CS_{\bu,M}$ and $\CS_{\bu,m}$. More precisely, let us denote with $\Im(H^{q}(\CS_{\bu,m})\rightarrow H^{q}(\CS_{\bu,M}))$ the image of $H^{q}(\CS_{\bu,m})$ into $H^{q}(\CS_{\bu,M})$ induced by the inclusion of complexes $i:\CS_{\bu,m}\rightarrow \CS_{\bu,M}$. We then have:

\begin{corol}
\label{finiteness}
In the setting of Th. \ref{TH}, the vector space $H^{q}(\CS'_{\bu})$ is finite dimensional and moreover 
$$\begin{aligned}
\dim(H^{q}(\CS_{\bu,M}))+&\dim(H^{q-1}(\CS_{\bu,M}))-\dim(\Im(H^{q-1}(\CS_{\bu,m})\rightarrow H^{q-1}(\CS_{\bu,M})))\\ 
&\leq\dim(H^{q}(\CS'_{\bu}))\\
&\leq \dim(H^{q}(\CS_{\bu,M}))+\dim(H^{q-1}(\CS_{\bu,M})).
\end{aligned}$$
\end{corol}

\begin{proof}
Since $\CS_{\bu,M}$ is a Fredholm complex, we know that $H^{q}(\CS_{\bu,M})$ is finite dimensional for every $q$ and thus, by Th. \ref{TH}, we know that 
$H^{q}(\CS'_{\bu})$ is finite dimensional and $$\dim(H^{q}(\CS'_{\bu}))\leq\dim(H^{q}(\CS_{\bu,M}))+\dim(H^{q-1}(\CS_{\bu,M})).$$ We are left to show that 
\begin{multline}
\nonumber
\dim(H^{q}(\CS'_{\bu}))\geq\\
 \dim(H^{q}(\CS_{\bu,M}))+\dim(H^{q-1}(\CS_{\bu,M}))-\dim(\Im(H^{q-1}(\CS_{\bu,m})\rightarrow H^{q-1}(\CS_{\bu,M}))).
\end{multline}
 In order to prove the above inequality, we claim that 
\begin{equation}
\label{claim}
\dim(\mathcal{H}^{q-1}_{m}(\CS_{\bu,M},\CS_{\bu,m}))\leq \dim(\Im(H^{q-1}(\CS_{\bu,m})\rightarrow H^{q-1}(\CS_{\bu,M}))).
\end{equation} 
Let us assume  for the moment the validity of \eqref{claim}. By Th. \ref{TH} we have $$\dim(H^{q}(\CS'_{\bu}))=\dim(H^{q}(\CS_{\bu,M}))+\dim(H^{q-1}(\CS_{\bu,M}))-\dim( \mathcal{H}^{q-1}_{m}(\CS_{\bu,M},\CS_{\bu,m}))$$ and consequently 
\begin{multline}
\nonumber
\dim(H^{q}(\CS'_{\bu}))\geq\\
 \dim(H^{q}(\CS_{\bu,M}))+\dim(H^{q-1}(\CS_{\bu,M}))-\dim(\Im(H^{q-1}(\CS_{\bu,m})\rightarrow H^{q-1}(\CS_{\bu,M})))
\end{multline}
 as required. Finally, we  prove \eqref{claim}.\\ First, note that $\Im(H^{q-1}(\CS_{\bu,m})\rightarrow H^{q-1}(\CS_{\bu,M}))$ can be described in the following manner, which is immediate to verify:
\begin{equation}
\label{alternative}
\Im(H^{q-1}(\CS_{\bu,m})\rightarrow H^{q-1}(\CS_{\bu,M}))=\frac{\ker(d_{q-1,m})}{\Im(d_{q-2,M})\cap\DS(d_{q-1,m})}.
\end{equation} 
Now let $0\neq \psi\in \mathcal{H}^{q-1}_{m}(\CS_{\bu,M},\CS_{\bu,m})$, let $[\psi]$ be the corresponding class in $H^{q-1}(\CS_{\bu,m})$, and finally, let $i_*[\psi]$ the image of $[\psi]$ in $H^{q-1}(\CS_{\bu,M})$ through the map $i_*:H^{q-1}(\CS_{\bu,m})\rightarrow H^{q-1}(\CS_{\bu,M})$ induced by the inclusion of complexes.\\ If $i_*[\psi]$ were trivial in $\Im(H^{q-1}(\CS_{\bu,m})\rightarrow H^{q-1}(\CS_{\bu,M}))$ then, by \eqref{alternative}, there would exists $\zeta\in \DS(d_{q-2,M})$ such that $d_{q-2,M}\zeta=\psi$. This is clearly not possible because $\psi\in \ker(\delta_{q-1,m})$, $\overline{\Im(d_{q-2,M})}$ is orthogonal to $\ker(\delta_{q-1,m})$ and we assumed that $0\neq \psi$. Therefore, we can conclude that the map $\mathcal{H}^{q-1}_{m}(\CS_{\bu,M},\CS_{\bu,m})\rightarrow \Im(H^{q-1}(\CS_{\bu,m})\rightarrow H^{q-1}(\CS_{\bu,M}))$ given by $\psi\mapsto i_*[\psi]$ is injective.
\end{proof}

Finally, we have the following corollaries

\begin{corol}
\label{rough-estimate}
In the setting of Th. \ref{TH} we have
$$\dim(H^{q}(\CS_{\bu,M}))\leq\dim(H^{q}(\CS'_{\bu})).$$
\end{corol}
\begin{proof}
This follows by Cor. \ref{finiteness} by noticing that $$\dim(H^{q-1}(\CS_{\bu,M}))-\dim(\Im(H^{q-1}(\CS_{\bu,m})\rightarrow H^{q-1}(\CS_{\bu,M})))\geq 0.$$
\end{proof}

\begin{corol}
\label{simply}
In the setting of Th. \ref{TH}, if in addition $\mathcal{H}_m^{q-1}(\CS_{\bu,M},\CS_{\bu,m})=\{0\}$, then, we have $$H^{q}(\CS'_{\bu})\cong H^{q}(\CS_{\bu,M})\oplus H^{q-1}(\CS_{\bu,M}).$$ In particular, the above isomorphism holds true whenever $$\Im(H^{q-1}(\CS_{\bu,m})\rightarrow H^{q-1}(\CS_{\bu,M}))=\{0\}.$$
\end{corol}

\begin{proof}
This follows by Th. \ref{TH} and \eqref{claim}.
\end{proof}

\begin{corol}
\label{simply-r}
In the setting of Th. \ref{TH}, if in addition:
\begin{enumerate}
\item $H^{q-1}(\CS_{\bu,m})=\{0\}$, then, $H^{q}(\CS'_{\bu})\cong H^{q}(\CS_{\bu,M})\oplus H^{q-1}(\CS_{\bu,M})$;
\item $H^{q-1}(\CS_{\bu,M})=\{0\}$, then, $H^{q}(\CS'_{\bu})\cong H^{q}(\CS_{\bu,M})$.
\end{enumerate}
\end{corol}

\begin{proof}
The conclusions follow by Th. \ref{TH} by noticing that both the above assumptions imply the vanishing of $\Im(H^{q-1}(\CS_{\bu,m})\rightarrow H^{q-1}(\CS_{\bu,M}))$ and consequently the vanishing of $\mathcal{H}_m^{q-1}(\CS_{\bu,M},\CS_{\bu,m})$.
\end{proof}

\section{Relative harmonic cohomology decomposition for Fredholm complexes}
\label{s7}

Following the lines of the previous section, this section is devoted to the study of the cohomology of the complex of relative harmonic vectors of the pair $(\CS_{\bu,M},\CS_{\bu,m})$ under the additional assumption that both $\CS_{\bu,M}$ and $\CS_{\bu,m}$ are Fredholm complexes. As we will see in a moment this assumption will make possible to substantially  refine Prop. \ref{morphism-b}. So we start with the following

\begin{ass} From now on we assume that both $\CS_{\bu,M}$ and $\CS_{\bu,m}$  are Fredholm complexes.
\end{ass}

Let us now consider the vector space $$\frac{\ker(P_q)\cap \Im(d_{q-1,m})}{\Im(d''_{q-1})}\oplus \mathcal{H}^{q-1}_{m}(\CS_{\bu,M},\CS_{\bu,m})$$
with 
$$\mathcal{H}^{q-1}_{m}(\CS_{\bu,M},\CS_{\bu,m}):=\ker(d_{q-1,m})\cap \ker(\de_{q-1,m}).$$ 

As above we have the inclusion 
$$\mathcal{H}^{q-1}_{m}(\CS_{\bu,M},\CS_{\bu,m})\hookrightarrow H^{q-1}(\CS_{\bu,M})$$ given by $$v\mapsto [v]$$ with $[v]$ the class induced by $v$ in $H^{q-1}(\CS_{\bu,m})$. 
Let us define 
$$\overline{\de}_q:\frac{\ker(P_q)\cap \Im(d_{q-1,m})}{\Im(d''_{q-1})}\oplus \mathcal{H}^{q-1}_{m}(\CS_{\bu,M},\CS_{\bu,m})\rightarrow H^{q-1}(\CS_{\bu,M})$$ as  the morphism given by  
\begin{equation}
\label{def-b}
\overline{\de}_q([u]\oplus v):=\check{\de}_{q}([u])+[v]
\end{equation}
for any $[u]\in \frac{\ker(P_q)\cap \Im(d_{q-1,m})}{\Im(d''_{q-1})}$ and $v\in \mathcal{H}^{q-1}_{m}(\CS_{\bu,M},\CS_{\bu,m})$. Note that by Prop. \ref{morphism-b} $\check{\de}_{q}([u])$ is a cohomology class in $H^{q-1}(\CS_{\bu,M})$ and, by the above discussion, $[v]$ is also a  cohomology class in $H^{q-1}(\CS_{\bu,M})$. Hence on the right-hand side of \eqref{def-b} we have the sum of two cohomology classes in $H^{q-1}(\CS_{\bu,M})$. The next theorem is the crucial step toward the main result of this section.

\begin{theo}
\label{crucial-b}
Let $\CS_{\bu,M}$ and $\CS_{\bu,m}$ be two Fredholm complexes as in \eqref{M}-\eqref{m}. Then, the morphism of vector spaces:
$$\overline{\de}_q:\frac{\ker(P_q)\cap \Im(d_{q-1,m})}{\Im(d''_{q-1})}\oplus \mathcal{H}^{q-1}_{m}(\CS_{\bu,m},\CS_{\bu,m})\rightarrow H^{q-1}(\CS_{\bu,M})$$ defined above is an isomorphism.
\end{theo}
\begin{proof}
Let $w\in \ker(P_q)\cap \Im(d_{q-1,m})$, $v\in \mathcal{H}^{q-1}_{m}(\CS_{\bu,M},\CS_{\bu,m})$ and $z\in \DS(d_{q-2,M})$ such that $\delta_{q,M}w+v=d_{q-2,M}z$. Thanks to the Hodge decomposition of Lemma \ref{newHodge}, we get that $v=0$ and $\delta_{q,M}w=d_{q-2,M}z$. Therefore, $\overline{\de}_{q,M}$ is injective if and only if 
$$\check{\de}_q:\frac{\ker (P_q)\cap 
\Im (d_{q-1,m})}{\Im(d''_{q-1})}
\to \frac{\ker(d_{q-1,M})}{\Im(d_{q-2,M})}$$ 
is injective. 
Let us show now that $\check{\de}_{q}$ is actually injective.  Take $a\in \ker (P_q)\cap
\Im (d_{q-1,m})$ and assume that $\de_{q,M} a\in \Im( d_{q-2,M})$. 
Then, we have to show that $a\in \Im (d''_{q-1})$. Since $a\in \Im (d_{q-1,m})$, by applying first the Hodge decomposition
$$
\HF^{q-1}=\left(\ker (d_{q-1,m})\cap \ker( \de_{q-1,M})\right)\oplus \Im( d_{q-2,m}) \oplus \Im (\de_{q,M})
$$  and then the Hodge decomposition
$$
\HF^{q}=\left(\ker (d_{q,m})\cap \ker( \de_{q,M})\right)\oplus \Im( d_{q-1,m}) \oplus \Im (\de_{q+1,M})
$$ 
we find an element $v\in \DS(\de_{q,M})\cap \Im(d_{q-1,m})$ such that $d_{q-1,m}(\delta_{q,M}v)=a.$ Moreover, by denoting $u:=\delta_{q,M}v$, we get $u\in \DS(P_{q-1})$ and $P_{q-1}u=\delta_{q,M}a$. By assumption we know that  there exists  $c\in \DS(d_{q-2,M})$ such that $d_{q-2,M}c=\de_{q,M}a$   and consequently $P_{q-1} u=d_{q-2,M}c$. Using the Hodge decomposition
\[
\HF^{q-2}=\left(\ker (d_{q-2,m})\cap \ker (\de_{q-2,M})\right)\oplus \Im( d_{ q-3,m}) \oplus \Im (\de_{q-1,M}),
\]
we can assume without loss of generality that  
\[
c\in \DS(\d_{q-2,M})\cap \Im \de_{q-1,M}.
\]
Therefore, $c=\de_{q-1,M}x$, for some $x\in\DS(\de_{q-1,M})$. Repeating the same argument, we can find $y\in \DS(\d_{q-2,m})$, such that
\[
c=\de_{q-1,M}x=\de_{q-1,M}(\d_{q-2,m}y).
\]

Since the domain of the minimal extensions are contained in the domains of the maximal extensions, we have
\[
P_{q-1} u=d_{q-2,M}c=d_{q-2,M}(\de_{q-1,M}(\d_{q-2,m}y))=P_{q-1}(\d_{q-2,m}y).
\]

Set $z=u-\d_{q-2,m}y$. Then, on one side
\[
P_{q-1} z=0,
\]
and on the other
\[
a=d_{q-1,m}u=d_{q-1,m}(u- d_{q-2,m}y)=d_{q-1,m}z,
\]
i.e. $a\in \Im(d''_{q-1})$, as desired. Now we proceed by showing that $\tilde{\de}_q$ is surjective. Let $u\in \ker(d_{q-1,M})$. Thanks to the Hodge decomposition of Lemma \ref{newHodge}, we can write $u=u_0+d_{q-2,M}u_1+\de_{q,M}u_2$ with $u_0\in \mathcal{H}^{q-1}_{\min}(\CS_{\bu,M},\CS_{\bu,m})$. Moreover, by applying the Hodge decomposition to $\HF_q$: 
$$
\HF^{q}=\left(\ker (d_{q,m})\cap \ker( \de_{q,M})\right)\oplus \Im( d_{q-1,m}) \oplus \Im (\de_{q+1,M})
$$ 
we obtain that $u_2=d_{q-1,m}\gamma$. Note that $d_{q-1,m}\gamma\in \DS(P_q)$ and $P_q(d_{q-1,m}\gamma)=0$. Indeed, it is obvious that $d_{q-1,m}\gamma\in \ker(d_{q,m})$. Furthermore, $d_{q-1,m}\gamma=u_2\in \DS(\de_{q,M})$ and  $\de_{q,M}u_2=u-u_0-d_{q-2,M}u_1\in \ker(d_{q-1,M})$. Therefore, $d_{q-1,m}\gamma\in \DS(P_q)$ and $P_q(d_{q-1,m}\gamma)=d_{q-1,M}(u-u_0-d_{q-2,M}u_1)=0$. 
Summarizing, we have 
$$d_{q-1,m}\gamma\in \ker(P_{q})\cap \Im(d_{q-1,m})\quad \mathrm{and}\quad \check{\de}_q([d_{q-1,m}\gamma])=[u-u_0]\in H^{q-1}(\CS_{\bu,M}).$$ 

We can thus conclude that 
$$\overline{\de}_q([d_{q-1,m}\gamma]\oplus u_0)=[u]\in H^{q-1}(\CS_{\bu,M})$$ and eventually that 
$$\overline{\de}_q:\frac{\ker(P_q)\cap \Im(d_{q-1,m})}{\Im(d''_{q-1})}\oplus \mathcal{H}^{q-1}_{m}(\CS_{\bu,M},\CS_{\bu,m})\rightarrow H^{q-1}(\CS_{\bu,M})$$ 
is an isomorphism of vector spaces.
\end{proof}


\begin{theo}
\label{TH-b}
Let $\CS_{\bu,M}$ and $\CS_{\bu,m}$ be two Fredholm complexes as in \eqref{M}-\eqref{m}. Let $H^{q}(\CS''_{\bu})$ be the cohomology  groups defined in \eqref{homology-b}. Then, we have an isomorphism of vector spaces
$$H^{q}(\CS''_{\bu})\oplus \mathcal{H}^{q-1}_{m}(\CS_{\bu,M},\CS_{\bu,m})\cong H^{q}(\CS_{\bu,m})\oplus H^{q-1}(\CS_{\bu,M}).$$ 
\end{theo}

\begin{proof}
Thanks to Prop. \ref{firststep-b} and Th. \ref{crucial-b} we have the following chain of isomorphisms:
$$
\begin{aligned}
&H^{q}(\CS''_{\bu})\oplus \mathcal{H}^{q-1}_{m}(\CS_{\bu,M},\CS_{\bu,m})\cong\\
&\cong (\ker (d_{q,M})\cap \ker( \de_{q,m}))\oplus \frac{\overline{\Im(d_{q-1,M})}\cap\ker(P_q)}{\Im(d'_{q-1})}\oplus \mathcal{H}^{q-1}_{m}(\CS_{\bu,M},\CS_{\bu,m})\\
&\cong H^q(\CS_{\bu,m})\oplus H^{q-1}(\CS_{\bu,M}).
\end{aligned}
$$
\end{proof}

We collect now various consequences of Th. \ref{TH-b} that follow by arguing as in the proofs of Cor. \ref{finiteness}, Cor. \ref{simply} and Cor. \ref{simply-r}. 

\begin{corol}
\label{finiteness-rel}
In the setting of Th. \ref{TH-b} the vector space $H^{q}(\CS''_{\bu})$ is finite dimensional and moreover 
$$\begin{aligned}
\dim(H^{q}(\CS_{\bu,m}))+&\dim(H^{q-1}(\CS_{\bu,M}))-\dim(\Im(H^{q-1}(\CS_{\bu,m})\rightarrow H^{q-1}(\CS_{\bu,M})))\\ 
&\leq\dim(H^{q}(\CS''_{\bu}))\\
&\leq \dim(H^{q}(\CS_{\bu,m}))+\dim(H^{q-1}(\CS_{\bu,M})).
\end{aligned}$$
\end{corol}

\begin{corol}
\label{simply-rel}
In the setting of Th. \ref{TH-b} if in addition $\mathcal{H}_m^{q-1}(\CS_{\bu,M},\CS_{\bu,m})=\{0\}$, then we have $$H^{q}(\CS''_{\bu})\cong H^{q}(\CS_{\bu,M})\oplus H^{q-1}(\CS_{\bu,M}).$$ In particular the above isomorphism holds true whenever $$\Im(H^{q-1}(\CS_{\bu,m})\rightarrow H^{q-1}(\CS_{\bu,M}))=\{0\}.$$
\end{corol}

\begin{corol}
\label{simply-s}
In the setting of Th. \ref{TH} if in addition:
\begin{enumerate}
\item $H^{q-1}(\CS_{\bu,m})=\{0\}$ then $H^{q}(\CS'_{\bu})\cong H^{q}(\CS_{\bu,m})\oplus H^{q-1}(\CS_{\bu,M})$;
\item $H^{q-1}(\CS_{\bu,M})=\{0\}$ then $H^{q}(\CS'_{\bu})\cong H^{q}(\CS_{\bu,m})$.
\end{enumerate}
\end{corol}

Interestingly the difference between $\dim(H^{q}(\CS'_{\bu}))$ and $\dim(H^{q}(\CS''_{\bu}))$ is fully determined by the cohomology of the complexes $\CS_{\bu,m}$ and $\CS_{\bu,M}$. More precisely we have 

\begin{corol}
\label{difference}
In the setting of Th. \ref{TH-b} we have $$\dim(H^{q}(\CS'_{\bu}))-\dim(H^{q}(\CS''_{\bu}))=\dim(H^{q}(\CS_{\bu,M}))-\dim(H^{q}(\CS_{\bu,m}))$$
\end{corol}

\begin{proof}
This follows at once by Th. \ref{TH} and Th. \ref{TH-b}.
\end{proof}

\section{Geometric applications}
\label{s8}

In this last part we collect various applications of the previous results to smoothly Thom-Mather stratified spaces. In order to make this section as self-contained as possible, we start with a  concise introduction to smoothly Thom-Mather stratified spaces and intersection cohomology.

\subsection{Riemannian manifolds and $L^2$-cohomology}
\label{s8.1}

 The aim of this section is to  recall briefly some basic notions about $L^2$-spaces and $L^2$-cohomology. Let $(M,g)$  be an open and possibly incomplete Riemannian manifold of dimension $m$ and let $\mathrm{dvol}_g$ be the one-density associated to $g$.  We consider $M$ endowed with the corresponding Riemannian measure.  A $k$-form $\omega$ is said {\em measurable} if, for any trivialization $(U,\phi)$ of $\Lambda^kT^*M$, $\phi(\omega|_U)$ is given by a $k$-tuple of measurable functions.  Given a measurable $k$-form $\omega$ the pointwise norm $|\omega|_g$ is defined as $|\omega|_g:= (g(\omega,\omega))^{1/2}$, where with a little abuse of notation we still label by $g$ the metric induced by $g$ on $\Lambda^kT^*M$. Then we can define $L^2\Omega^k(M,g)$ as the  space  of measurable  $k$-forms such that    $$\|\omega\|_{L^2\Omega^k(M,g)}:=\left(\int_{M}|\omega|_g^2\mathrm{dvol}_g\right)^{1/2}<\infty.$$ It is well known that we have a Hilbert space whose inner product is $$\langle \omega, \eta \rangle_{L^2\Omega^k(M,g)}:= \int_{M}g(\omega,\eta)\mathrm{dvol}_g$$ and such that $\Omega^k_c(M)$,  the space of smooth $k$-forms with compact support,  is a dense subspace. Consider now the de Rham complex of smooth forms with compact support: $$\cdot\cdot\cdot\stackrel{d_{k-1}}{\rightarrow} \Omega^k_c(M)\stackrel{d_k}{\rightarrow} \Omega_c^{k+1}(M)\stackrel{d_{k+1}}{\rightarrow}\cdot\cdot\cdot$$ We want to turn the above complex into a Hilbert complex and to do so we look at $d_k$ as an unbounded, densely defined and closable operator  acting between $L^2\Omega^k(M,g)$ and $L^2\Omega^{k+1}(M,g)$. In general $d_k$ admits several closed extensions. For our purposes we recall now the definitions of the maximal and minimal one. Let $$\delta_{k+1}:\Omega_c^{k+1}(M)\rightarrow \Omega^{k}_c(M)$$ be the formal adjoint of $d_k$. We recall that 
$\delta_{k+1}$ is the first order differential operator uniquely determined by the fact that $$\int_Mg(d_k\alpha,\beta)\mathrm{dvol}_g=\int_Mg(\alpha,\delta_{k+1}\beta)\mathrm{dvol}_g$$ for every $\alpha\in \Omega_c^{k}(M)$ and $\beta\in \Omega_c^{k+1}(M)$. Then, the domain of the {\em maximal extension} of $d_k:L^2\Omega^k(M,g)\longrightarrow L^2\Omega^{k+1}(M,g)$ is defined as
\begin{multline}
\nonumber\DS(d_{k,\max}):=\\
\{\omega\in L^2\Omega^k(M,g): \text{there is}\ \eta\in L^2\Omega^{k+1}(M,g)\ \text{such that} \int_{M}g(\omega,\delta_{k+1}\phi)\mathrm{dvol}_g\\
 =\int_{M}g(\eta,\phi)\mathrm{dvol}_g\ \text{for each}\ \phi\in \Omega^{k+1}_c(M)\}.\ \text{In this case we put}\ d_{k,\max}\omega=\eta.
\end{multline}

In other words the maximal extension of $d_k$ is the one defined in the {\em distributional sense}. The domain of the {\em minimal extension} of $d_k:L^2\Omega^k(M,g)\longrightarrow L^2\Omega^{k+1}(M,g)$ is defined as
\begin{multline}
\nonumber \DS(d_{k,\min}):=\\
\{\omega\in L^2\Omega^k(M,g)\ \text{such that there is a sequence}\ \{\omega_i\}\in \Omega_c^k(M)\ \text{with}\ \omega_i\rightarrow \omega\\ 
   \text{in}\ L^2\Omega^k(M,g)\ \text{and}\ d_k\omega_i\rightarrow \eta\ \text{in}\ L^2\Omega^{k+1}(M,g)\ \text{to some }\ \eta\in L^2\Omega^{k+1}(M,g)\}.
\end{multline}
In this case we put $d_{k,\min}\omega=\eta.$ Briefly the minimal extension of $d_k$ is the closure of $\Omega^{k}_c(M)$ under the graph norm of $d_k$. Clearly $\DS(d_{k,\min})\subset \DS(d_{k,\max})$ and $d_{k,\max}\omega=d_{k,\min}\omega$ for any $\omega\in \DS(d_{k,\min})$. It is easy to verify that if $\omega\in \DS(d_{k,\max/\min})$ then $d_{k,\max/\min}\omega\in \DS(d_{k+1,\max/\min})$ and the corresponding compositions are identically zero, that is $d_{k+1,\max}\circ d_{k,\max}\equiv 0$ and $d_{k+1,\min}\circ d_{k,\min}\equiv 0$. Summarizing we obtained two Hilbert complexes: $$L^2_{\bullet,\max}:=(L^2\Omega^{\bullet}(M,g),d_{\bullet,\max})\quad \mathrm{and}\quad L^2_{\bullet,\min}:=(L^2\Omega^{\bullet}(M,g),d_{\bullet,\min}).$$ 

The $L^2$-maximal/minimal de Rham cohomology of $(M,g)$ is defined as $$H^k_{2,\max/\min}(M,g):=\ker(d_{k,\max})/\Im(d_{k-1,\max/\min})$$ while the reduced $L^2$-maximal/minimal de Rham cohomology of $(M,g)$ is defined as 
$$\overline{H}^k_{2,\max/\min}(M,g):=\ker(d_{k,\max})/\overline{\Im(d_{k-1,\max/\min})}.$$
Furthermore, it is not difficult to check that $$(d_{k,\max/\min})^*=\delta_{k+1,\min/\max}.$$ We continue now by pointing out the following $L^2$-Hodge decompositions:

\begin{prop}
Let $(M,g)$ be a possibly incomplete Riemannian manifold. Then $$L^2\Omega^k(M,g)=\left(\ker(d_{k,\min})\cap \ker(\delta_{k,\min})\right)\oplus\left(\overline{\Im(d_{k-1,\max})}+\overline{\Im(\delta_{k+1,\max})}\right).$$
If $L^2_{\bullet,\max}$ is a Fredholm complex then
$$L^2\Omega^k(M,g)=\left(\ker(d_{k,\min})\cap \ker(\delta_{k,\min})\right)\oplus\left(\Im(d_{k-1,\max})+\Im(\delta_{k+1,\max})\right).$$
\end{prop}
\begin{proof}
The above proposition follows by applying lemma \ref{new-Hodge-weak} and \ref{newHodge} to $L^2(M,g)$ and the Hilbert complexes $L^2_{\bullet,\max}$ and $L^2_{\bullet,\min}$.
\end{proof}
 
Following the first two sections of this paper we can consider the operators $$P_q:=d_{q-1,\max}\circ \delta_{q,\max}+\delta_{q+1,\max}\circ d_{q,\max}$$ defined on their natural domains, and the complexes of $L^2$-harmonic forms and relative $L^2$-harmonic forms that we denote with $(L^{2'}_{\bullet},d'_{\bullet})$ and $(L^{2''}_{\bullet},d''_{\bullet})$, respectively. In the remaining part of this paper we will exhibit examples of Riemannian manifolds $(M,g)$ such that the complexes $(L^{2'}_{\bullet},d'_{\bullet})$ and $(L^{2''}_{\bullet},d''_{\bullet})$ satisfy the requirements of Th. \ref{TH} and Th. \ref{TH-b} and such that the corresponding cohomology group have a purely topological interpretations or, in the worst case, their dimensions have purely topological upper and lower bounds.\\

\noindent We conclude now this section with the following remark that highlight the fact that the complexes $(L^{2'}_{\bullet},d'_{\bullet})$ and $(L^{2'}_{\bullet},d'_{\bullet})$ produce a new cohomology only when $d_{\bullet,\max}\neq d_{\bullet,\min}$. More precisely we have

 \begin{rem}
\label{r1}
Let $(M,g)$ be a Riemannian manifold of dimension $m$ such that $d_{q,\max}=d_{q,\min}$ and $d_{q-1,\max}=d_{q-1,\min}$ for some $q\in \{0,...,m\}$. 
Passing to the adjoints we obtain  $\delta_{q,\max}=\delta_{q,\min}$ and $\delta_{q+1,\max}=\delta_{q+1,\min}$. Let $$\Delta_{q,\mathrel{abs}}:=d_{q-1,\max}\circ \delta_{q,\min}+\delta_{q+1,\min}\circ d_{q,\max}$$ and $$\Delta_{q,\mathrel{rel}}:=d_{q-1,\min}\circ \delta_{q,\max}+\delta_{q+1,\max}\circ d_{q,\min}$$ be the Laplacians that in degree $q$ are induced by the maximal/minimal $L^2$ de Rham complex, respectively. We then have
$$P_{q}=\Delta_{q,\mathrel{abs}}=\Delta_{q,\mathrel{rel}}$$ which obviously implies that $\ker(P_q)=\ker(\Delta_{q,\mathrm{abs}})=\ker(\Delta_{q,\mathrm{rel}})$. Since $$\ker(\Delta_{q,\mathrm{abs}})=\ker(d_{q,\max})\cap \ker(\delta_{q,\min})\quad\quad \ker(\Delta_{q,\mathrm{rel}})=\ker(d_{q,\min})\cap \ker(\delta_{q,\max})$$ we deduce that both the operators $d'_q$ and $d''_q$ are the zero operator. Similarly in degree $q-1$ we have $$P_{q-1}=d_{q-2,\max}\circ \delta_{q-1,\max}+\delta_{q,\min}\circ d_{q-1,\max}=d_{q,\max}\circ \delta_{q-1,\max}+\delta_{q,\max}\circ d_{q-1,\min}.$$ Since $\delta_{q,\min}\circ d_{q-1,\max}=\delta_{q,\max}\circ d_{q-1,\min}$ and $\Im(d_{q,\max}\circ \delta_{q-1,\max})$ is orthogonal to $\Im(\delta_{q,\min}\circ d_{q-1,\max})$ we can conclude that $\ker(P_{q-1})\subset \ker(d_{q-1,\max})=\ker(d_{q-1,\min})$ and thus also in this case we have that both the operators $d'_q$ and $d''_q$ are the zero operator. Summarizing, if $d_{q,\max}=d_{q,\min}$ and $d_{q-1,\max}=d_{q-1,\min}$  then $$H^{q}(L^{2'}_{\bullet})=H^{q}_{2,\max}(M,g)=H^{q}_{2\min}(M,g)=H^{q}(L^{2''}_{\bullet}).$$
\end{rem}

\subsection{Applications to compact manifolds with boundary}
\label{s8.2}

Let $\overline{M}$ be a compact manifold with boundary and let $M:\overline{M}\setminus \partial \overline{M}$ be the interior of $\overline{M}$. Finally let $\overline{g}$ be an arbitrarily fixed Riemannian metric on $\overline{M}$ and let $g:=\overline{g}|_{M}$. We are in position for the next

\begin{theo}
\label{boundary}
In the above setting we have the following isomorphisms:
$$H^q(L^{2'}_{\bullet})\cong H^{q-1}_{dR}(\overline{M})\oplus H^q_{dR}(\overline{M})$$
and
$$H^q(L^{2''}_{\bullet})\cong H^{q-1}_{dR}(\overline{M})\oplus H^q_{dR}(\overline{M},\partial\overline{M}).$$
\end{theo}

\begin{proof}
According to \cite[Th. 4.1]{BL} we know that $$H^q_{2,\max}(M,g)\cong H^q_{dR}(\overline{M})\quad \mathrm{and}\quad H^q_{2,\min}(M,g)\cong H^q_{dR}(\overline{M},\partial\overline{M}).$$ In particular, since $\overline{M}$ is compact, we can apply Th. \ref{TH} and Th. \ref{TH-b}  to conclude that 
$$H^q(L^{2''}_{\bullet})\oplus \mathcal{H}^{q-1}_m(L^2_{\bullet,\max}, L^2_{\bullet,\min})\cong H^{q-1}_{dR}(\overline{M})\oplus H^q_{dR}(\overline{M},\partial\overline{M})$$ and $$H^q(L^{2''}_{\bullet})\oplus\mathcal{H}^{q-1}_m(L^2_{\bullet,\max}, L^2_{\bullet,\min})\cong H^{q-1}_{dR}(\overline{M})\oplus H^q_{dR}(\overline{M},\partial\overline{M}).$$ Therefore, we are left to prove that $$\mathcal{H}^{q-1}_m(L^2_{\bullet,\max}, L^2_{\bullet,\min})=\{0\}.$$  First, we remind  that  $$\mathcal{H}^{q-1}_{m}(L^2_{\bullet,\max}, L^2_{\bullet,\min}):=\ker(d_{q-1,\min})\cap \ker(\delta_{q-1,\min}).$$
Let now $\alpha\in \mathcal{H}^{q-1}_m(L^2_{\bullet,\max}, L^2_{\bullet,\min})$. Then, there exist sequences $\{\phi_j\}_{j\in \mathbb{N}}\subset \Omega_c^{k}(M)$ and $\{\psi_j\}_{j\in \mathbb{N}}\subset \Omega_c^{k}(M)$ such that 
\begin{enumerate}
\item $\phi_j\rightarrow \alpha$ in $L^2\Omega^{q-1}(M,g)$ and $\psi_j\rightarrow \alpha$ in $L^2\Omega^{q-1}(M,g)$, as $j\rightarrow +\infty$;
\item $d_{q-1}\phi_j\rightarrow 0$ in $L^2\Omega^{q}(M,g)$ and $\delta_{q-1}\psi_j\rightarrow 0$ in $L^2\Omega^{q-2}(M,g)$, as $j\rightarrow +\infty$;
\end{enumerate}
Let $N$ be the double of $\overline{M}$. Then $N$ contains a submanifold with boundary diffeomorphic to $\overline{M}$ that with a little abuse of notation we still denote with $\overline{M}$. Let $h$ be a Riemannian metric on $N$ such that $h|_{M}=g$ and let $\alpha'\in L^2\Omega^{q-1}(N,h)$ be defined as $\alpha'|_{M}=\alpha$ and $\alpha'|_{N\setminus M}=0$. Then we can look at $\{\phi_j\}_{j\in \mathbb{N}}$ and $\{\psi_j\}_{j\in\mathbb{N}}$ as sequences with the following properties:
\begin{enumerate}
\item $\{\phi_j\}_{j\in \mathbb{N}}\subset \Omega^{q-1}(N)$ and $\{\psi_j\}_{j\in \mathbb{N}}\subset \Omega^{q-1}(N)$ 
\item $\phi_j\rightarrow \alpha'$ in $L^2\Omega^{q-1}(N,h)$ and $\psi_j\rightarrow \alpha'$ in $L^2\Omega^{q-1}(N,h)$, as $j\rightarrow +\infty$;
\item $d_{q-1}\phi_j\rightarrow 0$ in $L^2\Omega^{q}(N,h)$ and $\delta_{q-1}\psi_j\rightarrow 0$ in $L^2\Omega^{q-2}(N,h)$, as $j\rightarrow +\infty$;
\end{enumerate}
Let $\Delta_{q-1}$ be the Hodge Laplacian of $(N,h)$ acting on forms of degree $q-1$. The above properties imply that $\Delta_{q-1}\alpha'=0$ in the weak sense and since $\Delta_{q-1}$ is elliptic we can conclude that $\alpha'$ is smooth and $\Delta_k\alpha'=0$ in the classical sense. The unique continuation principle \cite[Th. 1.8]{Kazd} implies now that $\alpha'$ vanishes on $N$ since it vanishes on $N\setminus M$. Thus, we can conclude that $\alpha$ vanishes on $M$ and so $\mathcal{H}^{q-1}_m(L^2_{\bullet,\max}, L^2_{\bullet,\min})=\{0\}$, as required.
\end{proof}

\begin{corol}
In the setting of Th. \ref{boundary} we have the following equality:
$$\dim(H^q(L^{2'}_{\bullet}))-\dim(H^q(L^{2''}_{\bullet}))=\dim(H^q_{dR}(\overline{M}))-\dim(H^q_{dR}(\overline{M},\partial{\overline{M}})).$$
\end{corol}
\begin{proof}
This follows by Th. \ref{boundary} and Cor. \ref{difference}.
\end{proof}

Joining our result with the main result of \cite{CDGM} we get the following
\begin{corol}
In the setting of Th. \ref{boundary} denote by $(\mathrm{Harm}_{\bullet},d_{\bullet})$ the subcomplex of $(L^{2'}_{\bullet},d_{\bullet})$ given by the smooth harmonic forms on $\overline{M}$. Then the natural inclusion $$
j_\bu:(\mathrm{Harm}_{\bullet},d_{\bullet})\rightarrow (L^{2'}_{\bullet},d_{\bullet})$$ is a quasi-isomorphism of complexes.
\end{corol}
\begin{proof}
This follows by Th. \ref{boundary} and \cite[Th. 1]{CDGM}.
\end{proof}

\subsection{Applications to Thom-Mather stratified spaces}
\label{s8.3}

Since it will be used in the definition below we start by recalling that, given a topological space $Z$, $C(Z)$ stands for the cone over $Z$ that is, $C(Z)=Z\times [0,2)/\sim$ where $(p,t)\sim (q,r)$ if and only if $r=t=0$. We have now the following

\begin{defi}   
\label{thom}
 A smoothly Thom-Mather stratified space $X$  of dimension $m$  is a metrizable, locally compact, second countable space which admits a locally finite decomposition into a union of locally closed strata $\mathfrak{G}=\{Y_{\alpha}\}$, where each $Y_{\alpha}$ is a smooth, open and connected manifold, with dimension depending on the index $\alpha$. We assume the following:
\begin{enumerate}
\item[(i)] If $Y_{\alpha}$, $Y_{\beta} \in \mathfrak{G}$ and $Y_{\alpha} \cap \overline{Y}_{\beta} \neq \emptyset$ then $Y_{\alpha} \subset \overline{Y}_{\beta}$
\item[(ii)]  Each stratum $Y$ is endowed with a set of control data $T_{Y} , \pi_{Y}$ and $\rho_{Y}$ ; here $T_{Y}$ is a neighbourhood of $Y$ in $X$ which retracts onto $Y$, $\pi_{Y} : T_{Y} \rightarrow Y$
is a fixed continuous retraction and $\rho_{Y}: T_{Y}\rightarrow [0, 2)$  is a continuous  function in this tubular neighbourhood such that $\rho_{Y}^{-1}(0) = Y$. Furthermore, we require that if $Z \in \mathfrak{G}$ and $Z \cap T_{Y}\neq \emptyset$  then $(\pi_{Y} , \rho_{Y} ) : T_{Y} \cap Z \rightarrow Y\times [0,2) $ is a proper smooth submersion.
\item[(iii)] If $W, Y,Z \in \mathfrak{G}$, and if $p \in T_{Y} \cap T_{Z} \cap W$ and $\pi_{Z}(p) \in T_{Y} \cap Z$ then
$\pi_{Y} (\pi_{Z}(p)) = \pi_{Y} (p)$ and $\rho_{Y} (\pi_{Z}(p)) = \rho_{Y} (p)$.
\item[(iv)] If $Y,Z \in \mathfrak{G}$, then
$Y \cap \overline{Z} \neq \emptyset \Leftrightarrow T_{Y} \cap Z \neq \emptyset$ ,
$T_{Y} \cap T_{Z} \neq \emptyset \Leftrightarrow Y\subset \overline{Z}, Y = Z\ or\ Z\subset \overline{Y} .$
\item[(v)]  For each $Y \in \mathfrak {G}$, the restriction $\pi_{Y} : T_{Y}\rightarrow Y$ is a locally trivial fibration with fibre the cone $C(L_{Y})$ over some other stratified space $L_{Y}$ (called the link over $Y$ ), with atlas $\mathcal{U}_{Y} = \{(\phi,\mathcal{U})\}$ where each $\phi$ is a trivialization
$\pi^{-1}_{Y} (U) \rightarrow U \times C(L_{Y} )$, and the transition functions are stratified isomorphisms  which preserve the rays of each conic
fibre as well as the radial variable $\rho_{Y}$ itself, hence are suspensions of isomorphisms of each link $L_{Y}$ which vary smoothly with the variable $y\in U$.
\item[(vi)] For each $j$ let $X_{j}$ be the union of all strata of dimension less or equal than $j$, then 
$$X\setminus X_{m-1}\ \mathrm{is\ dense\ in}\ X$$
\end{enumerate}
\end{defi}

The \emph{depth} of a stratum $Y$ is largest integer $k$ such that there exists a chain of strata $Y=Y_{k},...,Y_{0}$ such that $Y_{j}\subset \overline{Y_{j-1}}$ for $1\leq j\leq k.$ A stratum of maximal depth is always a closed subset of $X$.  The  maximal depth of any stratum in $X$ is called the \emph{depth of $X$} as stratified spaces. Note that if $\mathrm{depth}(X)=1$ then for any singular stratum $Y\in \mathfrak{G}$ the corresponding link $L_Y$ is a smooth manifold. Consider now the filtration
\begin{equation}
X = X_{m}\supset X_{m-1}\supset X_{m-2}\supset...\supset X_{0}.
\label{pippo}
\end{equation}
 We refer to the open subset $X\setminus X_{m-2}$ as the regular set of $X$ while the union of all other strata is the singular set of $X$,
$$\mathrm{reg}(X):=X\setminus \mathrm{sing}(X)\quad \text{with}\quad \mathrm{sing}(X):=\bigcup_{Y\in \mathfrak{G}, \mathrm{depth}(Y)>0 }Y. $$
The dimension of $X$ is by definition the dimension of $\mathrm{reg}(X)$. Given two smoothly Thom-Mather  stratified spaces  $X$ and $X'$,  a stratified isomorphism between them is a
homeomorphism $F: X\rightarrow X'$ which carries the open strata of $X$ to the open strata of $X'$
diffeomorphically, and such that $\pi'_{F(Y )}\circ  F = F \circ \pi_Y$ , $\rho'_{F(Y)}\circ F=\rho_Y$  for all $Y\in \mathfrak{G}(X)$.
For more details, properties and comments we refer to \cite{ALMP}, \cite{BHS}, \cite{RrHS},  \cite{JMA}. 
We recall in addition that important examples of smoothly Thom-Mather stratified spaces are provided by manifolds with boundary, complex projective varieties  and quotient spaces $M/G$, with $M$ a manifold and $G$ a Lie group acting properly on $M$.\\ 
As a next step we introduce the class of smooth Riemmanian metrics on $\mathrm{reg}(X)$  we will work with. The definition is given by induction on the depth of $X$.  We label by $\hat{c}:=(c_2,...,c_m)$ a $(m-1)$-tuple of non negative real numbers and we recall that  two Riemannian metrics $g$ and $h$ on a manifold $M$  are said to be \emph{quasi-isometric}, briefly $g\sim h$, if there exists a real number $c>0$ such that $c^{-1}h\leq g\leq ch$.

\begin{defi}
\label{iter}
Let $X$ be a smoothly Thom-Mather-stratified space and let $g$ be a Riemannian metric on $\mathrm{reg}(X)$. If $\mathrm{depth}(X)=0$, that is $X$ is a smooth manifold, a $\hat{c}$-iterated conic metric is understood to be any smooth Riemannian metric on $X$. Suppose now that $\mathrm{depth}(X)=k$ and that the definition of $\hat{c}$-iterated conic metric is given in the case $\mathrm{depth}(X)\leq k-1$; then we call a smooth Riemannian metric $g$ on $\mathrm{reg}(X)$ a  $\hat{c}$-\emph{iterated conic metric}  if it satisfies the following properties:
\begin{itemize}
\item Let $Y$ be a stratum of $X$ such that $Y\subset X_{i}\setminus  X_{i-1}$; by   definition \ref{thom} for each $q\in Y$ there exist an open neighbourhood $U$ of $q$ in $Y$ such that 
$$
\phi:\pi_{Y}^{-1}(U)\longrightarrow U\times C(L_{Y})
$$
 is a stratified isomorphism; in particular, 
$$
\phi:\pi_{Y}^{-1}(U)\cap \mathrm{reg}(X)\longrightarrow U\times \mathrm{reg}(C(L_{Y}))
$$
is a smooth diffeomorphism. Then, for each $q\in Y$, there exists one of these trivializations $(\phi,U)$ such that $g$ restricted on $\pi_{Y}^{-1}(U)\cap \mathrm{reg}(X)$ satisfies the following properties:
\begin{equation} 
\label{yhn}
(\phi^{-1})^{*}(g|_{\pi_{Y}^{-1}(U)\cap \mathrm{reg}(X)})\sim  d r^2+h_{U}+r^{2c_{m-i}}g_{L_{Y}}
\end{equation}
 where $m$ is the dimension of $X$, $h_{U}$ is a Riemannian metric defined over $U$ and  $g_{L_{Y}}$ is a $(c_2,...,c_{m-i-1})$-\emph{iterated edge metric}  on $\mathrm{reg}(L_{Y})$, $dr^2+h_{U}+r^{2c_{m-i}}g_{L_{Y}}$ is a Riemannian metric of product type on $U\times \mathrm{reg}(C(L_{Y}))$ and with $\sim$ we mean \emph{quasi-isometric}. 
\end{itemize}
\end{defi} 

Note that if $Y\subset X_{m}\setminus X_{m-1}$ then \eqref{yhn} simplifies to 
$$(\phi^{-1})^{*}(g|_{\pi_{Y}^{-1}(U)\cap \mathrm{reg}(X)})\sim  d r^2+h_{U}$$ because in this case the link is a point. The existence of this kind of Riemannian metrics is unobstructed, as follows from the next

\begin{prop}
Let $X$ be a smoothly Thom-Mather-stratified space of dimension $m$. For any $(m-1)$-tuple  of positive numbers $\hat{c}=(c_2,...,c_m)$, there exists  a smooth Riemannian metric on $\mathrm{reg}(X)$ which is a $\hat{c}$-iterated edge metric.
\end{prop}

\begin{proof}
See \cite{BHS} or \cite{ALMP}, the latter in the case $\hat{c}=(1,...,1,...,1)$.
\end{proof}

 As we will  see in a moment the importance of this class of metrics stems from its deep connection with the topology of $X$. In particular the maximal/minimal $L^2$-cohomology groups of $\mathrm{reg}(X)$ with respect  to a $\hat{c}$-iterated conic metric are isomorphic to certain intersection cohomology  groups of $X$.\\

\noindent We are going now to provide a very succinct introduction to intersection cohomology. We will be very brief and we refer to the seminal papers of  Goresky  and  MacPherson \cite{GoMac}-\cite{GoMac2} and to the  monographs \cite{Bana}, \cite{Friedman-book} and \cite{KiWoo} for  in-depth treatments of this subject. First we start by recalling the notions of perversity and general perversity.

\begin{defi}
\label{perv}
Let $X$ be a smoothly Thom-Mather stratified space. A general perversity  on $X$ is any function
\begin{equation}
\label{ujn}
 p:\{Singular\ Strata\ of\ X\}\rightarrow \mathbb{Z}.
\end{equation}
If $X_{m-1}=X_{m_2}$, $p$ depends only on the codimension of the stratum and in addition $$p(2)=0\ and\ p(k)\leq p(k+1)\leq p(k)+1$$ with $k\in \{2,3,4,...\}$ then we simply call it perversity.
\end{defi}

\noindent Example of perversities are the zero perversity $0(k)=0$, the top perversity $t(k)=k-2$, the upper middle perversity $\overline{m}(k)=[\frac{k-1}{2}]$ and the lower middle perversity $\underline{m}(k)=[\frac{k-2}{2}]$, with $[\bullet]$ denoting the integer part of $\bullet$. Note that $\underline{m}(k)=\overline{m}(k)$ when $k$ is even. Given a (general) perversity $p$, the {\em dual (general) perversity} is defined as $q:=t-p.$\\

\noindent Let us consider now the standard $i$-simplex $\Delta_{i}\subset \mathbb{R}^{i+1}$.  The $j$-skeleton of $\Delta_{i}$ is the set of $j$-subsimplices. Given a smoothly Thom-Mather stratified space $X$ we say that a singular $i$-simplex in $X$, i.e. a continuous map $\sigma:\Delta_{i}\rightarrow X$, is $p$-allowable if
$$ 
\sigma^{-1}(X_{m-k}-X_{m-k-1})\subset \{(i-k+p(k))-\mathrm{skeleton\ of}\ \Delta_{i}\}\ \mathrm{for\ all}\ k\geq 2.
$$
Note that to each singular stratum $Y\subset X$ $p$ assigns the value corresponding to $\mathrm{cod}(Y)$. Given a field $F$, the elements of the space $I^{p}S_{i}(X,F)$ are defined as the finite linear combinations with coefficients in $F$ of singular $i$-simplex $\sigma:\Delta_{i}\rightarrow X$ such that $\sigma$ and $\partial \sigma$ are both $p$-allowable. It is clear that $(I^{p}S_{i}(X,F), \partial_{i})$ is a complex and the  singular intersection homology groups with respect to the perversity $p$, $I^{p}H_{i}(X,F)$, are defined as  the homology groups of this complex.\\

We show now how to associate a general perversity to a given $\hat{c}$-iterated conic metric.

\begin{defi} 
\label{metric-perversity}
Let $X$ be a smoothly stratified space with a Thom-Mather stratification and let $g$ be a $\hat{c}$-iterated conic metric on $\mathrm{reg}(X)$. Then the general perversity $p_{g}$ associated to $g$ is:
\begin{equation}p_{g}(Y):= Y\longmapsto  \left\{
\begin{array}{lll}
0 & l_{Y}=0 \\
\frac{l_{Y}}{2}+[[\frac{1}{2c_{m-i}}]] & l_{Y}\ even,\ l_{Y}\neq 0\ and\ Y\subset X_{m-i}\setminus X_{m-i-1} \\
\frac{l_{Y}-1}{2}+[[\frac{1}{2}+\frac{1}{2c_{m-i}}]]& l_{Y}\ odd\ and\ Y\subset X_{m-i}\setminus X_{m-i-1}
\end{array}
\right.
\end{equation}
with $l_{Y}=dimL_{Y}$ and, given any real and positive number $x$, $[[x]]$ is the greatest integer strictly less than $x$. Besides $p_g$ we consider also the corresponding dual perversity $q_g$, that is $$q_g:=t-p_g$$ with $t$ the top perversity.
\end{defi}

Note that $p_g=\overline{m}$ and  $q_g=\underline{m}$ provided $\hat{c}=(1,...,1)$.\\

 We can now recall the isomorphism theorem between the  $L^2$-cohomology and intersection cohomology. For the definition as well as more details concerning the notion of stratified coefficient system we refer to \cite{IGP}.

\begin{theo}
\label{L2Int}
Let $X$ be a compact smoothly oriented Thom-Mather stratified space of dimension $m$ and let $g$ be a $\hat{c}$-iterated conic metric on $\mathrm{reg}(X)$. We have the following isomorphisms for each $k=0,...,m$:
$$H^k_{2,\max}(\mathrm{reg}(X),g)\cong I^{q_g}H^k(X,\mathcal{R}_0)$$
and 
$$H^k_{2,\min}(\mathrm{reg}(X),g)\cong I^{p_g}H^k(X,\mathcal{R}_0)$$
with $\mathcal{R}_0$ the stratified coefficient system given by the pair of local coefficient systems consisting of $(X-X_{n-1})\times \mathbb{R}$ over $X-X_{n-1}$ where the fibers $\mathbb{R}$ have the discrete topology  and the constant $0$ system on $X_{n-1}$.
\end{theo}

\begin{proof}
When $X_1=X_2$ and $\hat{c}=(1,....,1)$ this theorem was first proved by Cheeger \cite{JC}. In the case $X_1=X_2$ and $p_g$ is a perversity the proof is due to Nagase \cite{Masa}. Finally the proof of the above general form is given in \cite{FBe}.
\end{proof}

We point out that in Th. \ref{L2Int}, when $X_{m-1}=X_{m-2}$ and $p_g$ is a perversity, then the vector spaces on the right hand sides boil down to 
$I^{q_g}H^k(X,\mathbb{R})$ and $I^{q_g}H^k(X,\mathbb{R})$, see \cite[p. 110]{IGP}.

\begin{corol}
\label{closed}
In the setting of theorem \ref{L2Int}. For any $k=0,...,m$ 
 $$H^k_{2,\max}(\mathrm{reg}(X),g)\ \mathrm{and}\ H^k_{2,\min}(\mathrm{reg}(X),g)$$ are finite dimensional.
In particular the complexes $$(L^2\Omega^k(\mathrm{reg(X)},d_{k,\max}))\ \mathrm{and}\ (L^2\Omega^k(\mathrm{reg(X)},d_{k,\max}))$$ are Fredholm complexes.
\end{corol}

\begin{proof}
This follows immediately by the fact that $X$ is compact and thus $I^{p_g}H^k(X,\mathcal{R}_0)$ and $I^{q_g}H^k(X,\mathcal{R}_0)$ are  finite dimensional. 
\end{proof}

We are finally in the good position to apply Th. \ref{TH} and Th. \ref{TH-b} to a smoothly Thom-Mather stratified space endowed with a $\hat{c}$-iterated conic metric.

\begin{theo}
\label{Strati}
Let $X$ be a compact smoothly oriented Thom-Mather stratified space of dimension $m$ and let $g$ be a $\hat{c}$-iterated conic metric on $\mathrm{reg}(X)$. We then have the following isomorphisms:
$$H^q(L^{2'}_{\bullet})\oplus \mathcal{H}^{q-1}_m(L^2_{\bullet,\max}, L^2_{\bullet,\min})\cong I^{q_q}H^{q-1}(X,\mathcal{R}_0)\oplus I^{q_q}H^{q}(X,\mathcal{R}_0)$$
and
$$H^q(L^{2''}_{\bullet})\oplus \mathcal{H}^{q-1}_m(L^2_{\bullet,\max}, L^2_{\bullet,\min})\cong I^{q_q}H^{q-1}(X,\mathcal{R}_0)\oplus I^{p_q}H^{q}(X,\mathcal{R}_0).$$
\end{theo}

\begin{proof}
By Cor. \ref{closed}, we know that $(L^2\Omega^k(\mathrm{reg(X)},d_{k,\max}))$ and  $ (L^2\Omega^k(\mathrm{reg(X)},d_{k,\max}))$ are Fredholm complexes. Now the conclusion follows by applying Th. \ref{TH}, Th. \ref{TH-b} and Th. \ref{L2Int}. 
\end{proof}

By Cor. \ref{finiteness}, Cor. \ref{finiteness-rel} and Cor. \ref{difference} we have the the next

\begin{corol}
\label{corol}
In the setting of Th. \ref{Strati}, we have the following:
\begin{multline}
\nonumber
\dim(I^{q_g}H^{q}(X,\mathcal{R}_0))+\\
\dim(I^{q_g}H^{q-1}(X,\mathcal{R}_0)-\dim(\Im(I^{p_g}H^{q-1}(X,\mathcal{R}_0)\rightarrow I^{q_g}H^{q-1}(X,\mathcal{R}_0)))\\ 
\leq\dim(H^{q}(L^{2'}_{\bullet}))\\
\leq \dim(I^{q_g}H^{q}(X,\mathcal{R}_0))+\dim(I^{q_g}H^{q-1}(X,\mathcal{R}_0)
\end{multline}
and
\begin{multline}
\nonumber \dim(I^{p_g}H^{q}(X,\mathcal{R}_0))+\\
\dim(I^{q_g}H^{q-1}(X,\mathcal{R}_0)-\dim(\Im(I^{p_g}H^{q-1}(X,\mathcal{R}_0)\rightarrow I^{q_g}H^{q-1}(X,\mathcal{R}_0)))\\ 
\leq\dim(H^{q}(L^{2''}_{\bullet}))\\
\leq \dim(I^{p_g}H^{q}(X,\mathcal{R}_0))+\dim(I^{q_g}H^{q-1}(X,\mathcal{R}_0).
\end{multline}
Furthermore we have the following equality:
$$\dim(H^q(L^{2'}_{\bullet}))-\dim(H^q(L^{2''}_{\bullet}))=\dim(I^{q_g}H^q(X,\mathcal{R}_0))-\dim(I^{p_g}H^q(X,\mathcal{R}_0)).$$
\end{corol}

In the remaining part of this section we focus on the case of manifold with conical singularities. These are examples of smoothly Thom-Mather stratified space. As before let $\overline{M}$ a compact  manifold with boundary. Let us denote with $\partial \overline{M}$ and $M$ the boundary and the interior of $\overline{M}$, respectively. Let $g$ be any smooth symmetric section of $T^*\overline{M}\otimes T^*\overline{M}$ that restricts to a Riemannian metric on $M$ and such   that  there exists an open neighbourhood $U$ of $\overline{M}$ and a diffeomorphism  $\psi:  [0,1)\times \partial\overline{M}\rightarrow U$ such that $$\psi^*(g|_U)=dx\otimes dx+x^2h(x)$$ with $h(x)$ a family of Riemannian metrics on $\partial\overline{M}$ that depends smoothly on $x$ up to $0$. Finally let $X$ be the quotient space defined by $\overline{M}/\sim$ with $p\sim q$ if and only if $p$ and $q\in \partial\overline{M}$. It is immediate to check that $X$ becomes a compact smoothly Thom-Mather stratified space  with only one isolated singularity whose regular part is diffeomorphic to $M$. With a little abuse of notation we still denote with $g$ the Riemannian metric that $\mathrm{reg}(X)$ inherits from $(M,g)$. Clearly Def. \ref{iter} is satisfied by $(\mathrm{reg}(X),g)$. We will refer to $g$ as a metric with isolated conical singularities. We want to show now that is this specific setting Th. \ref{Strati} simplifies considerably. More precisely:

\begin{corol}
Let $X$ be a compact and oriented manifold with only isolated singularities of dimension $m$ and let $g$ be a conical metric on $\mathrm{reg}(X)$.
\begin{enumerate}
\item  If $m=2n$ is even or if $m=2n+1$ is odd and $H^n(\partial\overline{M})=\{0\}$ then  for each $q=0,...,m$ we have
$$H^q(L^{2'}_{\bullet})\cong I^{\underline{m}}H^{q}(X,\mathbb{R})\quad \mathrm{and}\quad H^q(L^{2''}_{\bullet})\cong I^{\overline{m}}H^{q}(X,\mathbb{R}).$$

\item If  $m=2n+1$ is odd and $H^n(\partial\overline{M})\neq \{0\}$ then we have
$$H^q(L^{2'}_{\bullet})\cong I^{\underline{m}}H^{q}(X,\mathbb{R})\quad \mathrm{and}\quad H^q(L^{2''}_{\bullet})\cong I^{\overline{m}}H^{q}(X,\mathbb{R})$$
for each $q\neq n+1$ and $$H^{n+1}(L^{2'}_{\bullet})\cong I^{\underline{m}}H^{n}(X,\mathbb{R}).$$

\end{enumerate}
\end{corol}

\begin{proof}
If $\dim X=2n$  or $\dim X=2n+1$ and $H^n(\partial\overline{M})=\{0\}$, then $d_{q,\max}=d_{q,\min}$ for each $q=0,...,m$, see \cite[Th. 3.7-3.8]{BLK}, and the conclusion follows by applying Rmk. \ref{r1} and Th. \ref{L2Int}. 
Let us consider the second point. Again by \cite[Th. 3.8]{BLK} we know that $d_{q,\max}=d_{q,\min}$ whenever $q\neq n$. Therefore, we can use Rmk. \ref{r1} to conclude that $$H^{q}(L^{2'}_{\bullet})=H^{q}_{2\max}(\mathrm{reg}(X),g)=H^{q}_{2\min}(\mathrm{reg}(X),g)=H^{q}(L^{2''}_{\bullet})$$ whenever $q<n$ or $q>n+1$. The conclusion now follows by applying Th. \ref{L2Int}. For the case $q=n$, we note that  $d_{n-1,\max}=d_{n-1,\min}$ and $\delta_{n-1,\max}=\delta_{n-1,\min}$. Therefore, in degree $n-1$ we have $$\mathcal{H}_m^{n-1}(L^{2}_{\bullet,\max},L^2_{\bullet,\min})=\ker(d_{n-1,\max})\cap \ker(\delta_{n-1,\min})\cong  H^{n-1}_{2,\max}(\mathrm{reg}(X),g).$$ Thus, the isomorphisms of Th. \ref{TH} and Th. \ref{TH-b} read in this case as $$H^{n}(L^{2'}_{\bullet})\oplus  H^{n-1}_{2,\max}(\mathrm{reg}(X),g)\cong   H^{n}_{2,\max}(\mathrm{reg}(X),g)\oplus H^{n-1}_{2,\max}(\mathrm{reg}(X),g)$$ and $$H^{n}(L^{2''}_{\bullet})\oplus  H^{n-1}_{2,\max}(\mathrm{reg}(X),g)\cong   H^{n}_{2,\min}(\mathrm{reg}(X),g)\oplus H^{n-1}_{2,\max}(\mathrm{reg}(X),g).$$ Consequently, $$H^{n}(L^{2'}_{\bullet})\cong H^n_{2,\max}(\mathrm{reg}(X),g)\quad \mathrm{and}\quad H^{n}(L^{2''}_{\bullet})\cong H^{n}_{2,\max}(\mathrm{reg}(X),g)$$ and the conclusion now follows by Th. \ref{L2Int}. Finally, to completes the proof of the second point, we tackle the case $q=n+1$. Since $\delta_{n,\max}=\delta_{n,\min}$  we have $$\mathcal{H}_m^n(L^{2}_{\bullet,\max},L^2_{\bullet,\min})=\ker(\delta_{n,\max})\cap \ker(d_{n,\min})\cong H^{n}_{2,\min}(\mathrm{reg}(X),g)\cong H^{n+1}_{2,\max}(\mathrm{reg}(X),g).$$ Thus the isomorphism of Th. \ref{TH}  reads in this case as $$H^{n+1}(L^{2'}_{\bullet})\oplus  H^{n+1}_{2,\max}(\mathrm{reg}(X),g)\cong   H^{n+1}_{2,\max}(\mathrm{reg}(X),g)\oplus H^{n}_{2,\max}(\mathrm{reg}(X),g)$$ and thus we get $$H^{n+1}(L^{2'}_{\bullet})\cong  H^{n}_{2,\max}(\mathrm{reg}(X),g).$$ Once again by applying Th. \ref{L2Int} we reach the desired conclusion.
\end{proof}

\begin{rem}
As a by-product of the above proof we also get that   $$H^{n+1}(L^{2''}_{\bullet})\oplus  H^{n}_{2,\min}(\mathrm{reg}(X),g)\cong   H^{n+1}_{2,\min}(\mathrm{reg}(X),g)\oplus H^{n}_{2,\max}(\mathrm{reg}(X),g).$$ However, this doesn't seem to lead to any further simplification.
\end{rem}

\bibliographystyle{plain}

\end{document}